\documentclass[12pt]{article}
\usepackage{amssymb,amsmath,amsthm}
\usepackage{hyperref}
\newcommand{\mr}[1]{\mathrm{#1}}

\newcommand{\eq}[1]{(\ref{#1})}

\newcommand{\Z}{{\bf Z}}

\newcommand{\F}{{\bf F}}

\newcommand{\qp}{{\bf Q}_p}
\newcommand{\zp}{{\bf Z}_p}
\newcommand{\ra}{\rightarrow}
\newcommand{\cs}{^{\times}}

\newcommand{\zpn}{\zeta_{p^n}}

\DeclareMathOperator{\Tr}{\mr{Tr}}

\DeclareMathOperator{\Gal}{Gal}

\newtheorem{theorem}{Theorem}[subsection]
\newtheorem{proposition}[theorem]{Proposition}
\newtheorem{lemma}[theorem]{Lemma}
\newtheorem{corollary}[theorem]{Corollary}

\newtheorem{theorem2}{Theorem}[section]
\newtheorem{proposition2}[theorem2]{Proposition}
\newtheorem{lemma2}[theorem2]{Lemma}

\theoremstyle{remark}
\newtheorem*{ack}{Acknowledgments}
\newtheorem{remark}[theorem]{Remark}
\newtheorem{example}[theorem]{Example}

\newtheorem{remark2}[theorem2]{Remark}

\begin{document}
\title{Galois module structure of local unit groups}
\author{Romyar T. Sharifi}
\date{}
\maketitle

\begin{abstract}
	We study the groups $U_i$
	in the unit filtration of a finite abelian extension $K$ of $\qp$, for an odd prime $p$.
	We determine explicit generators of the $U_i$ as modules over the 
	$\zp$-group ring of $\Gal(K/\qp)$.  
	We work in eigenspaces for powers of the
	Teichm\"uller character, first at the level of the field of
	norms for the extension of $K$ by $p$-power roots of unity 
	and then at the level of $K$.
\end{abstract}

\section{Introduction}
\label{intro}

Fix an odd prime $p$ and a finite unramified extension $E$ of $\qp$. 
We use $F_n$ to denote the field obtained from $E$ by adjoining to $E$ the $p^n$th roots of 
unity in an algebraic closure of $\qp$.  The $i$th unit group in the unit filtration of $F_n$
will be denoted by $U_{n,i}$.  The object of this paper is to describe generators of the 
groups $U_{n,i}$ as modules over the $\zp$-group ring of $G_n = \Gal(F_n/\qp)$.  We express these generators in terms of generators of the pro-$p$ completion $D_n$ of $F_n^{\times}$ as a Galois module.  
In fact, one consequence of our work is
a rather elementary proof of an 
explicit presentation of $D_n$ as such a module, as was proven by Greither \cite{greither} using Coleman theory.

Instead of working with all of $D_n$ at once, we find it easier to work with certain eigenspaces of it.
For this and several purposes, it will be useful to think of the Galois group $G_n$ as a
direct product of cyclic subgroups
$$ 
	G_n = \Delta \times \Gamma_n \times \Phi,
$$
where $\Delta \times \Gamma_n = \Gal(F_n/E)$ with $|\Delta| = p-1$ and $|\Gamma_n| = p^{n-1}$, and
$\Phi$ is isomorphic to $\Gal(E/\qp)$.  We then decompose $D_n$ into a direct sum of $p-1$ 
eigenspaces for powers of the Teichm\"uller character $\omega \colon \Delta \to \zp^{\times}$.
For any integer $r$, 
the $\omega^r$-eigenspace $D_n^{(r)}$ of $D_n$ is the subgroup of elements upon which 
$\sigma \in \Delta$ acts by left
multiplication by $\omega(\sigma)^r$.  
This definition depends only on $r$ modulo $p-1$, so we fix $r$ with $2 \le r \le p$.
Note that $D_n^{(r)}$ is a module over the
group ring $A_n = \zp[\Gamma_n \times \Phi]$.  In fact, as we shall see in Section \ref{eig}, 
the $A_n$-module $D_n^{(r)}$ has a generating set with just one element if $r \le p-2$, three elements 
if $r = p-1$, and two elements if $r = p$.

We will be interested in the $A_n$-module structure of the groups
$$
	V_{n,i}^{(r)} = D_n^{(r)} \cap U_{n,i}.
$$
It turns out that
$$
	V_{n,i}^{(r)} \supsetneq V_{n,i+1}^{(r)} = V_{n,i+2}^{(r)} = \cdots = V_{n,i+p-1}^{(r)}
$$ 
for all $i \equiv r \bmod p-1$ (see Lemma \ref{eigmod}), so we will consider only such
$i$ and set $V_{n,i} = V_{n,i}^{(r)}$.

Our main results, Theorems \ref{gensfin} and \ref{mingensfin}, provide a small set of at most $n+1$ generators of $V_{n,i}$ as an $A_n$-module and state that any proper generating subset of it has cocardinality $1$.  The elements of this set are written down explicitly as $A_n$-linear 
combinations of elements of the generators of $D_n^{(r)}$.  In Section \ref{specfin}, elements of a special form are constructed so as to lie as deep in the unit filtration as possible.  In Section \ref{genfin}, these are refined to elements of the same form that instead lie just deep enough to be in $V_{n,i}$, which are in turn the generators that we use.

It is convenient to work first in the field of norms $F$ of Fontaine-Wintenberger
for the tower of extensions $F_n$ of $E$.  This is a field of characteristic $p$, the multiplicative group of 
which is the inverse limit of the $F_n^{\times}$.  
We prove analogues of all of the above-mentioned results first at this infinite level, prior to applying them in descending to the level of $F_n$.  The fact that the $p$th power map is an automorphism of $F^{\times}$ simplifies some of the computations.  Moreover, the structure of the eigenspaces of the pro-$p$ completion of $F^{\times}$, which we study in Section \ref{eig}, is somewhat simpler than that of the $D_n^{(r)}$.   We construct special elements in the eigenspaces of the groups in the unit filtration in Section \ref{spec}, refine them in Section \ref{refined}, and prove generation and a minimality result in Section \ref{gen}.

We see a number of interesting potential applications for the results of this paper.
To mention just one, it appears to make possible the computation of the 
conductors of all degree $p^n$ 
Kummer extensions of $F_n$ in terms of the Kummer generator of the extension.  The problem
of making this computation, which was approached by the author in three much earlier papers,
has until now seemed beyond close reach in this sort of generality.

\begin{ack}
	The idea for this paper originated with the author's 1999 Ph.D. thesis, and
	initial computations were performed on an evening in June 2001 during a visit to the 
	University of Nottingham.  The author thanks Ivan
	Fesenko for his hospitality.  A short draft was written in 2002, when the author was supported 
	by an NSF Postdoctoral Research Fellowship.
	In August 2006, some additions were made, and that work was funded in 
	part by an NSERC Discovery Grant and the Canada Research Chairs program.  
	As the paper tripled in size in the summer of 2011, the author was 
	supported in part by NSF Grant DMS-0901526.  The author thanks Richard Gottesman for
	his interest in this work, which inspired him to finish the paper.
\end{ack}

\section{Preliminaries} \label{prelim}

\numberwithin{equation}{section}

We maintain the notation of the introduction and introduce some more.
Recall from \cite{winten} that the field of norms $F$ for the extension $F_{\infty} = \cup_n F_n$ of $E$ 
is a local field of characteristic $p$
with multiplicative group
$$
	F\cs = \lim_{\leftarrow} F_n\cs,
$$
the inverse limit being taken with respect to norm maps.  

Let $\zeta = (\zpn)_n$ be a norm
compatible sequence of $p$-power roots of unity, with $\zpn$ 
a primitive $p^n$th root of unity in $F_n$.  
Then
$$
	\lambda = 1-\zeta = (1-\zeta_{p^n})_n
$$
is a prime element of $F$.

For $m \ge n$, let $N_{m,n} \colon F_m \to F_n$ be the norm map.
Recall that the addition on $F$ is given by 
$$
	(\alpha+\beta)_n = \lim_{m \to \infty}  N_{m,n}(\alpha_m + \beta_m)
$$
for $\alpha = (\alpha_n)_n$ and $\beta = (\beta_n)_n$ in $F$. 
We fix an isomorphism of the residue field of $E$ (and thereby each $F_n$) with $\F_q$, with
$q$ the order of the residue field.  Using this, the field $\F_q$ is identified with a subfield of $F$
via the map that takes $\xi \in \F_q^{\times}$ to $(\tilde{\xi}^{p^{-n}})_n \in F^{\times}$, where 
$\tilde{\xi}$ is the $(q-1)$th root of unity in $E$ lifting $\xi$.
The field $F$ may then be identified with the field of Laurent series $\F_q((\lambda))$.

If $F_{\infty}$ is the union of the $F_n$, then $G = \Gal(F_{\infty}/\qp)$ acts as automorphisms 
on the field $F$.  As with $G_n$, we may decompose $G = \Gal(F_{\infty}/\qp)$  into a
direct product of procyclic subgroups
\[ G = \Delta \times \Gamma \times \Phi, \]
where $\Gal(F_{\infty}/E) = \Delta \times \Gamma$, the group $\Delta$
has order $p-1$, the group $\Gamma$ is isomorphic to $\zp$, and
$\Phi$ is isomorphic to $\Gal(E/\qp)$.
Let $\gamma$ denote the topological generator of $\Gamma$ such that
$\gamma(\zeta_{p^n}) = \zeta_{p^n}^{1+p}$ for all $n$.

The pro-$p$ completion $D$ of $F\cs$ decomposes
into a direct sum of eigenspaces for the powers of the Teichm\"uller
character $\omega$ on $\Delta$.  For an integer $r$,
we let $D^{(r)} = D^{\varepsilon_r}$, where $\varepsilon_r$ is the idempotent
\[
  \varepsilon_r = \frac{1}{p-1}
  \sum_{\delta \in \Delta} \omega(\delta)^{-r}\delta \in \zp[\Delta].
\]
For $i \ge 1$, let $U_i$ denote the $i$th group in the unit filtration of $F$. 
We then set 
\begin{eqnarray*} 
	V_i^{(r)} = U_i \cap D^{(r)} &\text{and}& (V_i^{(r)})' = V_i^{(r)}-V_{i+1}^{(r)}.
\end{eqnarray*}

The following is Lemma 2.3 of \cite{me-cond} (with $F_n$ replaced by $F$).

\begin{lemma2} \label{eigmod}
	We have $V_i^{(r)}/V_{i+p-1}^{(r)} \cong {\bf F}_q$ for every $i \ge 1$, and
  	$(V_i^{(r)})' \neq \varnothing$ if and only if $i \equiv r \bmod p-1$.
\end{lemma2}

From now on, we set $V_i = V_i^{(r)}$ and $V'_i = (V_i^{(r)})'$ if
$i \equiv r \bmod p-1$.  
As a consequence of Lemma \ref{eigmod}, an element $z \in V_i$
is determined modulo $\lambda^{i+p-1}$ by its expansion 
\begin{equation} \label{expansion}
	z \equiv 1 + \xi\lambda^i \bmod \lambda^{i+1}
\end{equation}
with  $\xi \in {\bf F}_q$.  

The following is Lemma 2.4 of \cite{me-cond} (with $F_n$ replaced by $F$).

\begin{lemma2} \label{moving}
	Let $z \in V'_i$.  If $p \nmid i$, then $z^{\gamma-1} \in V'_{i+p-1}$.  
	Otherwise, $z^{\gamma-1} \in V_{i+2(p-1)}$.
\end{lemma2}

We identify $\Lambda = \zp[[\Gamma]]$ with the power series ring $\zp[[T]]$
via the continuous, $\zp$-linear isomorphism that takes $\gamma-1$ to $T$, and we use 
additive notation to describe the action of $\zp[[T]]$ on $D$.  Ramification theory would
already have told us that $T \cdot V_i \subseteq V_{i+p-1}$ for all $i$. 
On the other hand, explicit calculation will yield the following two lemmas and proposition, which 
provide more precise information on how powers of $T$ move elements of $V_i$.

For $\xi \in \F_q^{\times}$, we let $V_i(\xi)$ denote the set of
$z \in V_i$ for which $z$ has an expansion
of the form in \eqref{expansion}.
We use $[ k ]$ to denote the smallest nonnegative integer congruent to $k \in \Z$ modulo $p$.

\begin{lemma2} \label{rtimes}
  Let $z \in V_i(\xi)$ for some $i$. 
  Then, for $0 \le j \le [ i ]$, we have
  $$
  	T^j z \in V_{i+j(p-1)}
	\left(\frac{[ i ]!}{([ i ]-j)!} \cdot \xi\right).
  $$
\end{lemma2}

\begin{proof}
  Note that
  \begin{equation} \label{lamgam}
    \lambda^{\gamma} = 1 - \zeta^{1+p} = 1 - (1-\lambda)(1-\lambda^p) =
    \lambda+\lambda^p-\lambda^{p+1}.
  \end{equation}
  Using this, we see, for any $i \ge 1$, that
  \begin{equation} \label{step2}
    (1+\xi\lambda^i)^{\gamma-1}
    \equiv 1 + i\xi\lambda^{i+p-1}\frac{1-\lambda}{1+\xi\lambda^i} 
    \bmod \lambda^{i+2p-2}.
  \end{equation}
  Hence,
  \begin{equation} \label{actbyT2}
    (1 + \xi\lambda^i)^{\gamma-1} \equiv 1 + i\xi\lambda^{i+p-1} \bmod 
    \lambda^{i+p}.
  \end{equation}
  Applying \eq{actbyT2} recursively, we obtain the result.
\end{proof}

\begin{lemma2} \label{repeat0}
  Let $z \in V_{pi-p+1}(\xi)$ for some $i \ge 2$.
  If $j$ is a nonnegative multiple of $p-1$, then
  $T^{j+1} z \in V_{p(i+j)}(\xi)$.
\end{lemma2}

\begin{proof}
	Let us begin by proving slightly finer versions of \eqref{step2} in two congruence classes
	of exponents modulo $p$.  
	For any $t \ge 1$, we have
	$$
		(1+\xi\lambda^{pt})^{\gamma-1} = 
		\frac{1+\xi\lambda^{pt}(1+\lambda^{p(p-1)}-\lambda^{p^2})^t}{1+\xi\lambda^{pt}}
		\equiv 1 \bmod \lambda^{p(t+p-1)}
	$$
	and
	\begin{align*}
		(1+\xi\lambda^{pt+1})^{\gamma-1} 
		&= 1 + \xi\lambda^{pt+1}
		\frac{\sum_{m=1}^{pt+1} \binom{pt+1}{m} (\lambda^{p-1}-\lambda^p)^m}{1+\xi\lambda^{pt+1}}
		\\ &\equiv 1 + \xi(\lambda^{p(t+1)} - \lambda^{p(t+1)+1}) 
		\bmod (\lambda^{p(t+p-1)+1},\lambda^{p(2t+1)+1}),
	\end{align*}	
	the latter congruence following from the fact that $p \mid \binom{pt+1}{m}$ for $2 \le m < p$.
	Via some obvious inequalities, we conclude that 	
	\begin{eqnarray} \label{actbyT0}
		(1+\xi\lambda^{pt})^{\gamma-1} &\equiv& 1 \bmod \lambda^{p(t+2)},\\
		\label{actbyT1}
		(1+\xi\lambda^{pt+1})^{\gamma-1} &\equiv& (1+\xi \lambda^{p(t+1)})(1-\xi \lambda^{p(t+1)+1}) \bmod 
		\lambda^{p(t+2)}.
	\end{eqnarray}
	
  	Let $x = 1+\xi\lambda^{pi-p+1}$.
  	Recursively applying \eqref{actbyT0} and \eqref{actbyT1}, we see that
  	$$
    		x^{(\gamma-1)^{k+1}}
    		\equiv (1+(-1)^k\xi\lambda^{p(i+k)})(1+(-1)^{k+1}\xi\lambda^{p(i+k)+1})
    		\bmod \lambda^{p(i+k+1)},
  	$$
  	for any positive integer $k$, as \eqref{actbyT2} 
  	implies that $U_{p(i+k)}^{\gamma-1} \subseteq U_{p(i+k+1)}$.
  	The result now follows by application of $\varepsilon_i$, since 
  	$z^{-1}x^{\varepsilon_i} \in V_{pi}$, $T^{j+1}x^{\varepsilon_i} \in V_{p(i+j)}(\xi)$, and 
	$T^{j+1}V_{pi} \subset V_{p(i+j+1)-1}$ by Lemma \ref{moving}.
\end{proof}

Let us use $\{k\}$ to denote the smallest nonnegative integer congruent to
$k \in \Z$ modulo $p-1$.
For $i \ge 1$ with $p \nmid i$, we define a monotonically-increasing function 
$\phi^{(i)} \colon \Z_{\ge 0} \ra \Z$ by $\phi^{(i)}(0) = i$ and
\begin{equation} \label{phidef}
   \phi^{(i)}(a) = pa + (i-[ i ])+ \{[ i ]-a\} 
\end{equation}
for $a \ge 1$.

\begin{proposition2} \label{repeat}
  Let $z \in V_i(\xi)$ for some $i \ge 2$
  with $p \nmid i$.  Then, for $j \ge 1$, we have
  $$
  	T^j z \in V_{\phi^{(i)}(j)}\left(\frac{[ i ]!}{
	\{[ i ] - j\}!}\xi\right).
  $$
\end{proposition2}

\begin{proof}
  	Lemma \ref{rtimes} implies that
 	$$
  	 	T^{[ i ] - 1} z \in V_{\phi^{(i)}([ i ] - 1)}([ i ]! \cdot \xi),
  	$$
  	and note that $\phi^{(i)}([ i ] - 1) \equiv 1 \bmod p$.
  	Set $k = \{[ i ]-j\}$.  Since $j+k-[ i]$ is divisible by $p-1$,
  	Lemma \ref{repeat0} then implies that
  	\begin{equation} \label{2ndTpow}
  		T^{j+k} z \in V_{\phi^{(i)}(j+k)}([ i ]! \cdot \xi).
  	\end{equation}
  	It follows from \eqref{phidef} that
  	$$
  		\phi^{(i)}(j+k) - i = p(j+k-[i]) + (p-1)[i],
  	$$
  	and so, given \eqref{2ndTpow}, Lemma \ref{moving} forces $T^l z \in V'_{\phi^{(i)}(l)}$
  	for all $l \le j+k$.  In particular, applying Lemma \ref{rtimes}
  	with $j$ replaced by $k$ and $z$ replaced by $T^j z$, we see that for
  	\eqref{2ndTpow} to hold, $T^jz$ must have the stated form.
\end{proof}

\begin{remark2}
	The obvious analogues of the
	 results of this section all hold at the level of $F_n$ for $n \ge 2$, with $\lambda$ replaced by
	$\lambda_n = 1-\zeta_{p^n}$.
	In fact, Lemmas \ref{eigmod} and \ref{moving} were originally proven in that setting
	in \cite{me-cond}.  That the other results hold breaks down to the fact that 
	$p$ is a unit times $\lambda_n^{p^{n-1}(p-1)}$ in $F_n$, which in particular tells us that \eqref{lamgam} can
	be replaced by
	$$
		\lambda_n^{\gamma} \equiv \lambda_n + \lambda_n^p - \lambda_n^{p+1} \bmod 
		\lambda_n^{p(p-1)+1}.
	$$
\end{remark2}

\section{The infinite level} \label{inflevel}

\numberwithin{equation}{subsection}

\subsection{Structure of the eigenspaces} \label{eig}

In this subsection, we fix choices of certain elements that will be used throughout the paper.
From now on, we let $\xi$ denote an element of $\F_q$ with $\Tr_{\Phi} \xi = 1$, the conjugates of which form a normal basis of $\F_q$ over $\F_p$.  
Let $\varphi \in \Phi$ denote the Frobenius element.  Let $N_{\Phi} \in \zp[\Phi]$ denote
the norm element.  Let $\zeta = (\zeta_{p^n})_n$ be
a norm-compatible system of primitive $p^n$th roots of unity as before.

Let $r$ be an integer satisfying $2 \le r \le p$.  If $2 \le r \le p-2$, we simply fix an element $u_r \in V_r(\xi)$.
In the case that $r = p-1$, generation of $D^{(p-1)}$ requires one additional element $\pi \in D^{(p-1)}$, a non-unit, chosen along with $u_{p-1} \in V_{p-1}(\xi)$ in the lemma which follows.
The case of $r = p$ shall require more work, but we will fix elements
$w \in V_1(-\xi)$ and $u_p \in V_p(\xi)$ as in Proposition \ref{newgen} below.

\begin{lemma} \label{pi}
	There exist elements $\pi \in D^{(p-1)}$ and $u_{p-1} \in V_{p-1}(\xi)$ such that $\pi^{\varphi} = \pi$ and 
	$\pi^{\gamma-1} = u_{p-1}^{N_{\Phi}}$.
\end{lemma}

\begin{proof}
	Set $\pi = \lambda^{\varepsilon_{p-1}}$, which satisfies $\pi^{\varphi} = \pi$
	and $\pi^{\gamma-1} \in V_{p-1}(1)$.  Since every unit is a norm in an unramified extension,
	there exists $u'_{p-1} \in D^{(p-1)}$ such that $(u'_{p-1})^{N_{\Phi}} = \pi^{\gamma-1}$,
	and such an element must lie in $V_{p-1}(\xi')$ for some $\xi' \in \F_q$ with $\Tr_{\Phi} \xi' = 1$.  
	Hilbert's Theorem 90 tells us that
	$\xi' = \xi + (\varphi-1)\eta$ for some $\eta \in \F_q$.  Let $z \in V_{p-1}(\eta)$, and set
	$u_{p-1} = u'_{p-1}z^{1-\varphi}$.  
\end{proof}

In fact, one could have chosen $u_{p-1} \in V_p(\xi)$ arbitrarily and then taken $\pi$ to satisfy the relations, as can be seen using the results of the following section.  

\begin{lemma} \label{defofelts}
	There exist elements $w \in V_1(-\xi)$ and $u_p \in V_p(\xi)$ with $w^{N_{\Phi}} = \zeta$ and
	$u_p^{\varphi-1} = w^{\gamma-1-p}$.
\end{lemma}

\begin{proof}
	First, local class field theory yields the existence of an element $w' \in D^{(p)}$ with 
	$(w')^{N_{\Phi}} = \zeta$.  
	Since $\zeta \in V_1(-1)$, 
	we must have $w' \in V_1(-\xi')$ for some $\xi' \in \F_q$ with $\Tr_{\Phi} \xi' = 1$.
	Since $\xi' = \xi + (\varphi-1)\eta$ for some $\eta \in \F_q$, we choose any 
	$y \in V_1(\eta)$, and then $w = w'y^{1-\varphi} \in V_1(-\xi)$ satisfies
	$w^{N_{\Phi}} = \zeta$ as well.
	
	Next, note that $(w^{\gamma-1-p})^{N_{\Phi}} = 1$, and so Hilbert's Theorem 90 allows us
	to choose an element $u_p' \in D^{(p)}$ with $(u_p')^{\varphi-1} = w^{\gamma-1-p}$.  
	A simple computation using \eqref{actbyT2} tells us that 
	$$
		w^{\gamma-1-p} \in V_p(\xi^p-\xi),
	$$
	and therefore $u_p' \in V_p(\xi + a)$ for some $a \in \F_p$.  We may then choose
	$z \in V_p(a)$ with $z^{\varphi} = z$ and take $u_p = u_p'z^{-1} \in V_p(\xi)$.
\end{proof}

We need slightly finer information on the relationship between $w$ and $u_p$ inside the
unit filtration, as found in the following proposition.

\begin{proposition} \label{newgen}
	There exist elements $w \in V_1(-\xi)$ and $u_p \in V_p(\xi)$ with $w^{N_{\Phi}} = \zeta$ and
	$u_p^{\varphi-1} = w^{\gamma-1-p}$ such that the element
  	$y = u_p w^{p\varphi^{-1}}$ lies in $V_{2p-1}(-\xi)$.
\end{proposition}

\begin{proof}
	For now, fix any choices of $u_p$ and $w$ as in Lemma \ref{defofelts}.
	We must have 
	$u_p = (1+\xi\lambda^p)^{\varepsilon_1}\alpha$ with $\alpha \in V_{2p-1}$
	and $w = (1-\xi\lambda)^{\varepsilon_1}\beta$ with $\beta \in V_p$.  Note that
	$$
		(1+\xi\lambda^p)^{\varphi-1} \equiv 1+(\xi^p-\xi)\lambda^p \bmod \lambda^{2p}
	$$
	and
	$$
		(1-\xi\lambda)^{\gamma-1-p} = 
		\frac{1-\xi(\lambda+\zeta\lambda^p)}{(1-\xi\lambda)(1-\xi^p\lambda^p)}
		\equiv 1+\left( \xi^p - \xi \frac{1-\lambda}{1-\xi\lambda}\right) \lambda^p \bmod \lambda^{2p}.
	$$
	We then have
	\begin{equation} \label{quotcomp}
		\frac{(1-\xi\lambda)^{\gamma-1-p}}{(1+\xi\lambda^p)^{\varphi-1}}
		\equiv 1 + \frac{\xi(1-\xi)}{1-\xi\lambda}\lambda^{p+1} \bmod \lambda^{2p}.
	\end{equation}
	We denote the quantity on the right-hand side of \eqref{quotcomp} by $\theta$.  
	By Lemma \ref{moving}, we must have
	$$
		\alpha^{\varphi-1} \theta^{-\varepsilon_1} = \beta^{\gamma-1-p} \in V_{3p-2}.
	$$
	On the other hand, by Lemma \ref{eigmod}, we have
	$$
		y\alpha^{-1} = 
		(1+\xi\lambda^p)^{\varepsilon_1}(1-\xi\lambda^p)^{\varepsilon_1}\beta^{p\varphi^{-1}}
		\in V_{3p-2},
	$$
	so in fact we have
	$y^{\varphi-1} \theta^{-\varepsilon_1} \in V_{3p-2}$.
	If we can show that $\theta^{\varepsilon_1} \in V_{2p-1}(\xi-\xi^p)$, we will then have
	$y \in V_{2p-1}(-\xi+a)$ for some $a \in \F_p$.  As in the proof of Lemma \ref{defofelts},
	we can then choose an element $z \in V_{2p-1}(a)$ with $z^{\varphi} = z$ and replace
	$u_p$ by $u_pz^{-1}$ to obtain the result.
	
	By Proposition \ref{repeat}, we see that to show that $\theta^{\varepsilon_1} \in V_{2p-1}(\xi-\xi^p)$, 
	it suffices to show that
	$$
		\theta^{\varepsilon_1(\gamma-1)^{p-1}} \in V_{p^2}(\xi^p-\xi).
	$$
	Since $p^2 \equiv 1 \bmod p-1$, for this, it suffices to show that
	$$
		\theta^{(\gamma-1)^{p-1}} \equiv 1 + (\xi^p-\xi)\lambda^{p^2} \bmod \lambda^{p^2+1}.
	$$
	This is a simple consequence of Lemma \ref{recurexp}, which follows.  That is, in the notation of said 
	lemma, Fermat's
	little theorem and the binomial theorem tell us that $d_{p-1,k} = -1$ for all positive 
	integers $k \le p-1$.
\end{proof}

\begin{lemma} \label{recurexp}
	For each positive integer $j \le p-1$, one has
	$$
		\left( 1 + \frac{\xi(1-\xi)}{1-\xi\lambda}\lambda^{p+1} \right)^{(\gamma-1)^j}
		\equiv 1 + \left( \sum_{k=1}^j d_{j,k} \xi^k (1-\xi) \right)  \lambda^{(j+1)p} \bmod
		\lambda^{(j+1)p+1},
	$$
	where 
	$$
		d_{j,k} = \sum_{h = 1}^k (-1)^{j+h} \binom{k}{h} h^j \in \F_p
	$$
	for positive integers $k \le j$.
\end{lemma}

\begin{proof}
	We make the expansion
	$$
		\theta = 1 + \frac{\xi(1-\xi)}{1-\xi\lambda}\lambda^{p+1} 
		\equiv \prod_{k=1}^{p-1}(1+\xi^k(1-\xi)\lambda^{p+k}) \bmod \lambda^{2p}.
	$$
	Since $U_s^{\gamma-1} \subseteq U_{s+p-1}$ for all $s$, 
	as follows from \eqref{actbyT2}, to compute
	$\theta^{(\gamma-1)^j}$ modulo $\lambda^{(j+1)p+1}$, it suffices to compute
	$(1+\xi^k(1-\xi)\lambda^{p+k})^{(\gamma-1)^j}$ modulo $\lambda^{(j+1)p+1}$.
	
	Fix a positive integer $k \le p-1$.  We claim that the coefficient of $\lambda^{(j+1)p}$
	in the expansion of $(1+\xi^k(1-\xi)\lambda^{p+k})^{(\gamma-1)^j}$ as a power series
	in $\F_q[[\lambda]]$ is $0$ if $j < k$ and $\xi^k(1-\xi)d_{j,k}$ if $j \ge k$.  
	As a consequence of \eqref{step2}, one sees that
	$$
		(1+\xi \lambda^t)^{\gamma-1} \equiv (1 + t\xi \lambda^{t+p-1})(1 - t\xi \lambda^{t+p})
		\bmod \lambda^{t+2p-2}
	$$
	for any $t \ge p-1$.  Using this and the finer congruence \eqref{actbyT1} 
	when possible, an induction yields that the expansion in question is determined by
	$$
		\prod_{m=0}^{\min(j,k)} 
		\prod_{(a_i) \in P_{j,k,m}} \left(1+ \xi^k(1-\xi) \frac{k!}{(k-m)!} 
		\prod_{i=1}^{j-m} a_i\cdot \lambda^{(j+1)p+k-m}\right)^{(-1)^{j-m}}
		\bmod \lambda^{(j+1)p+k+1},
	$$
	where
	$$
		P_{j,k,m} = \{ ((a_1, a_2, \ldots, a_{j-m}) \in \Z^{j-m} \mid
		k-m \le a_1 \le a_2 \le \cdots \le a_{j-m} \le k \}
	$$
	if $j > m$ and $P_{j,k,j} = \{0\}$, and we consider the empty
	product to be $1$.
	In particular, the coefficient in question is indeed $0$ for $j < k$ and is 
	$\xi^k(1-\xi)c_{j,k}$
	for $j \ge k$, where
	$$
		c_{j,k} = (-1)^{j-k} k! 
		\sum_{(a_i) \in P_{j,k,k}} \prod_{i=1}^{j-k} a_i.
	$$
	It remains to verify that $c_{j,k} = d_{j,k}$ for $j \ge k$.
	
	Let $D$ denote the differential operator $x\frac{d}{dx}$
	on $\F_p[x]$.  By the binomial theorem, we have that
	$$
		D^j((1-x)^k)|_{x=1} = \sum_{h=1}^k (-1)^h \binom{k}{h} h^j x^h \Big|_{x=1} = (-1)^j d_{j,k}.
	$$  
	On the other hand, repeated application of the product formula for the derivative yields that
	$$
		D^j((1-x)^k)|_{x=1} = (-1)^k\sum_{h=1}^{\min(j,k)} \frac{k!}{(k-h)!}
		\sum_{(a_i) \in P_{j,h,h}} \prod_{i=1}^{j-h} a_i  \cdot (x-1)^{k-h}x^h \Big|_{x=1}
		= (-1)^jc_{j,k}
	$$
	for all $j \ge k$, hence the result.
\end{proof}

In the next section, we will obtain the following very slight refinement of what is essentially 
a result of Greither's \cite[Sections 2-3]{greither} (see also \cite[Corollary 2.2]{me-cond}).

\begin{theorem} \label{presentation}
	For $r \le p-2$, the $A$-module $D^{(r)}$ is freely generated by any $u_r \in V_r(\xi)$.  
	The $A$-module $D^{(p-1)}$ has a presentation
	$$
    		D^{(p-1)} = \langle \pi,  u_{p-1} \mid  \pi^{\varphi} = \pi,\, 
      		u_{p-1}^{N_{\Phi}} = \pi^{\gamma-1}\rangle,
    	$$
    	for some $u_{p-1} \in V_{p-1}(\xi)$ and $\pi \in D^{(p-1)}$.   
	The $A$-module $D^{(p)}$ has a presentation
    	$$
      		D^{(p)} = \langle u_p, w \mid w^{\gamma-1-p} =
      		u_p^{\varphi-1} \rangle,
    	$$
   	for some $u_p \in V_p(\xi)$ and $w \in V_1(-\xi)$ such that $w^{N_{\Phi}} = \zeta$.
\end{theorem}

\subsection{Special elements} \label{spec}

Fix $r$ with $2 \le r \le p$, and define $\phi \colon \Z_{\ge 0} \to \Z$ by $\phi(a) = \phi^{(r)}(a)$ for $a \ge 1$. 
Set 
$$
	\delta =  \begin{cases} 0 & \text{if } 2 \le r \le p-1,\\
	1 & \text{if } r = p. \end{cases}
$$ 
For all $a \ge 1$, we have
$$
	\phi(a) = p(a+\delta) + \{r-\delta-a\},
$$
so $\phi(a)$ is the smallest integer that is at least $p(a+\delta)$ and congruent
to $r$ modulo $p-1$.

From now on, $i$ will be used solely to denote a positive integer congruent
to $r$ modulo $p-1$.  Let us use the notation $a \sim b$ to denote the equivalence
relation on $D^{(r)}$ given by $a \sim b$ if both $a$ and $b$ lie in $V_i(\xi)$ for some $i$ and $\xi \in \F_q^{\times}$.  We use additive notation for the action of $A = \zp[\Phi][[T]]$ on $D^{(r)}$.
We begin with the following useful lemma.

\begin{lemma} \label{basic}
	Let $j$ be a positive integer.
	\begin{enumerate}
		\item[a.] We have
		$$
			T^j u_r \in V_{\phi(j)}\left(\frac{[r]!}{\{r-\delta-j\}!}\xi\right).
		$$
		\item[b.] If $j \equiv r-\delta \bmod (p-1)$ so that
		$T^ju_r \sim pz$ for some $z \in D^{(r)}$, then
		$$
			T^j u_r - pz \in V_{\phi(j)+p-1}\left(-[r]!\xi\right).
		$$
	\end{enumerate}
\end{lemma}

\begin{proof}
	For $r < p$, part a is a direct consequence of Proposition \ref{repeat} and the fact that
	$u_r \in V_r(\xi)$.  For $r = p$, Proposition \ref{repeat} and the fact that $\phi = \phi^{(2p-1)}$ 
	on positive integers would tell us more directly
	that $T^j y \in V_{\phi(j)}(\frac{1}{\{-j\}!} \xi)$ for $j \ge 1$, for $y$ as in Proposition \ref{newgen}.
	Note, however, that 
	$$
		Tu_p = Ty - p\varphi^{-1}Tw \sim Ty,
	$$
	since $pT w \in V_{p^2}$.  This is also the key point of part b.  That is, we have
	$$
		T(T^j u_r - pz) \sim T^{j+1} u_r
	$$
	as $pTz \in V_{p\phi(j)}$ and 
	$$
		\phi(j+1) \le \phi(j)+2(p-1) < p\phi(j).
	$$  
	Since
	$$
		T^{j+1} u_r \in V_{\phi(j)+2(p-1)}([r]!\xi),
	$$
	a final application of Proposition \ref{repeat} tells us that $T^ju_r-pz$ had to be
	in the stated group.
\end{proof}

For a nonnegative integer $m$, let us define $\phi_m \colon \Z_{\ge 0} \ra \Z_{\ge 0}$ by 
$$
	\phi_m = p^m(\phi+1)-1.
$$  
We remark that
\begin{equation} \label{compose}
	p\phi_m = \phi \circ (\phi_m-\delta).
\end{equation}

From now on, we set $\rho = p\varphi^{-1}$ for brevity of notation.
We define special elements in the unit filtration of $D^{(r)}$.

\begin{theorem} \label{main}
	Let $m$ and $j$ be nonnegative integers.  
  	Define
  	\begin{equation*}
    		\alpha_{m,j} = \frac{1}{[r]!}\Bigg(\{r-\delta-j\}! \rho^m T^j 
		-  \sum_{k=1}^{m} \rho^{m-k} T^{\phi_{k-1}(j)-\delta}\Bigg)u_r,
  	\end{equation*}
	unless $j = 0$ and $r = p-1$, in which case we replace $\{r-\delta-j\}!$ with $-1$
	in the formula.
  	Then $\alpha_{m,j} \in V_{\phi_m(j)}(\xi)$.  
	Furthermore,
  	$$
    		(p^m b T^j + c)u_r \notin V_{\phi_m(j)+p-1}
  	$$
  	for all $b \in \zp[\Phi]$ with $b \not\equiv 0 \bmod p$ and $c \in T^{j+1}A$.
\end{theorem}

\begin{proof}
  	We work by induction, the case of $m = 0$ being Lemma \ref{basic}a, aside
	from the case $j = 0$, in which case it is simply the definition of $u_r$.
  	Assume we have proven the first statement for $m$. 
	Then
  	$$ 
		p \alpha_{m,j} \in V_{p\phi_m(j)}( \xi^p) 
	$$
  	and, using Lemma \ref{basic}a and \eq{compose}, we have
  	$$ 
		T^{\phi_{m}(j)-\delta} u_r \in V_{p\phi_m(j)}\bigl([r]!\xi\bigr). 
	$$
 	Lemma \ref{basic}b then tells us that
  	$$ 
		\alpha_{m+1,j} = \rho \alpha_{m,j} - \frac{1}{[r]!}
     		T^{\phi_m(j)-\delta} u_r \in V_{p\phi_m(j)+p-1}(\xi).  
	$$

  	Now assume the second statement is true for $m$.  (For $m = 0$, this is a consequence
  	of the fact that the conjugates of $\xi$ are $\F_p$-linearly independent.)
  	Suppose that
  	$$
    		\alpha = (p^{m+1}bT^j +c)u_r \in V'_i
  	$$
  	with $i \ge \phi_{m+1}(j)$, $b \in \zp[\Phi] - p\zp[\Phi]$ and $c \in T^{j+1}A$.
  	We write $c = (p c' + T^h \nu)u_r$ for some $c', \nu \in A$ with $\nu \not\equiv 0 \bmod (p,T)$
  	and $h \ge j+1$. 
	By induction, we have that
  	$$
		(p^mbT^j + c' )u_r \notin V_{\phi_m(j)+p-1}.
	$$ 
  	Since $\phi_{m+1}(j) = p\phi_m(j) + p-1$ and
	$\alpha \in V_{\phi_{m+1}(j)}$ by assumption, this forces
  	$$
		p(p^mbT^j + c') u_r \sim -T^h \nu u_r,
	$$
	which tells us by Lemma \ref{basic}a that $\phi(h) \le p\phi_m(j)$. 
	On the other hand, Lemma \ref{basic}b tells us that $\alpha \in V'_{\phi(h)+p-1}$,
	which forces $i = \phi_{m+1}(j)$.
\end{proof}

The second statement of Theorem \ref{main} insures, in particular, that $a u_r \neq 0$ for
all nonzero $a \in A$.  We therefore have the following corollary.

\begin{corollary} \label{freeness}
	The $A$-submodule of $D^{(r)}$ generated by $u_r$ is free.
\end{corollary}

In the exceptional case that $r = p$, we require additional elements.  First, we modify
the function $\phi_m$ for this $r$.  For nonnegative integers $m$ and $j$, we set
$\phi'_m(j) = \phi_m(j)$ unless $r = p$ and $j = p^l -1$ for some $l \ge 0$, in which case we set
$$
	\phi'_m(p^l-1) = p^{m+l+1}+p^{m+1}-1 = \phi_m(p^l-1) + p^m(p-1).
$$

\begin{theorem} \label{main1}
	Let $l$ and $m$ be nonnegative integers.
	Define
	$$
		\beta_{m,l} = \Bigg(\rho^m T^{p^l-1} + 
		\sum_{k=1}^{m} \rho^{m-k} T^{\phi'_{k-1}(p^l-1)-1}\Bigg)u_p+ 
		\rho^{m+l+1}w.
	$$
	Then $\beta_{m,l} \in V_{\phi'_m(p^l-1)}(-\xi)$.  Moreover, for any $j \ge 0$, 
	we have
	$$
		(p^m b T^j + c) u_p + d w \notin V_{\phi'_m(j)+p-1}
	$$
	for all $b \in \zp[\Phi] - p\zp[\Phi]$, $c \in T^{j+1}A$ and $d \in \zp[\Phi]$.
\end{theorem}

\begin{proof}
	The proof is similar to that of Theorem \ref{main}.
	Since Lemma \ref{basic}a and the definition of $w$ tell us that
	\begin{eqnarray*}
		T^{p^l-1}u_p \in V_{p^{l+1}}(\xi)
	&
	\text{and}
	&
		\rho^{l+1}w \in V_{p^{l+1}}(-\xi),
	\end{eqnarray*}
	Lemma \ref{basic}b yields
	$$
		\beta_{0,l} = \rho^{l+1} w + T^{p^l-1}u_p \in V_{p^{l+1}+p-1}(-\xi).
	$$
	For any $m \ge 0$, we have
	$$
		\beta_{m+1,l} = \rho \beta_{m,l} + T^{\phi'_m(p^l-1)-1}u_p.
	$$
	By induction and Lemma \ref{basic}a, we have 
	\begin{eqnarray*}
		\rho\beta_{m,l} \in V_{p\phi'_m(p^l-1)}(-\xi)  &\text{and}&
		T^{\phi'_m(p^l-1)-1}u_p \in V_{p\phi'_m(p^l-1)}(\xi).
	\end{eqnarray*}
	Since $\phi'_{m+1}(p^l-1) = p\phi'_m(p^l-1)+p-1$, that $\beta_{m+1,l} \in V_{\phi'_{m+1}(p^l-1)}(-\xi)$
	is just another application of Lemma \ref{basic}b.
	
	Let $j \ge 0$, $b \in \zp[\Phi]-p\zp[\Phi]$, $c \in T^{j+1}A$, and $d \in \zp[\Phi]$.
	First, suppose that $\alpha = (bT^j + c)u_p + dw \in V'_i$ for some $i \ge \phi_0'(j)$.
	Note that $(bT^j+c)u_p \in V'_{\phi(j)}$ by Lemma \ref{basic}a, while $dw \in V_{p^l}'$ for
	some $l \ge 0$.  Since $\alpha \in V_{\phi_0'(j)}$, we must have $p^l \ge \phi(j)$.  We then
	have $\alpha \in V'_{\phi(j)}$ unless $\phi(j) = p^l$.  This occurs if and only if $l \ge 1$ and
	$j = p^{l-1}-1$, in which case $\phi'_0(j) = p^l+p-1$.  
	For this to hold, we must have $(bT^j + c)u_p \sim -dw$.  Lemma \ref{basic}b then implies that 
	$\alpha \in V'_{p^l+p-1}$, so $i = \phi'_0(j)$ in all cases.
		
	Suppose now that
	$$
		\alpha = (p^{m+1}bT^j+c)u_p + dw \in V'_i
	$$ 
	for some $i \ge \phi'_{m+1}(j)$.
	Rewrite $c$ as $pc' + T^h\nu$ for some $h \ge j+1$ and 
	$c', \nu \in A$ with $\nu \not\equiv 0 \bmod (p,T)$.  If we are to have
	$\alpha \in V_p$, we may also write $d = pd'$ for some $d' \in \zp[\Phi]$.
	By induction, we have that
	$$
		(p^mbT^j + c') u_p + d' w \notin V_{\phi'_m(j)+p-1},
	$$
	and so in order that $\alpha \in V_{\phi'_{m+1}(j)}$, we must have that
	$$
		(p^{m+1}bT^j + pc') u_p + d w \sim -T^h\nu u_p
	$$
	which tells us using Lemma \ref{basic}a that $\phi(h) \le p\phi'_m(j)$.
	On the other hand, Lemma \ref{basic}b tells us that $\alpha \in V'_{\phi(h)+p-1}$,
	so we must have $i = \phi'_{m+1}(j)$.
\end{proof}

Theorem \ref{presentation} may now be proven as a consequence of the description of the
above elements and their place in the unit filtration.

\begin{proof}[Proof of Theorem \ref{presentation}]
	For $r \le p-1$, the union of the disjoint images of the functions $\phi_m$ is exactly the set of 
	positive integers
	congruent to $r$ modulo $p-1$.  Therefore, Theorem \ref{main} implies that there
	exists an element of the $A$-module generated by $u_r$ in $V_i(\xi)$ for each 
	$i \equiv r \bmod p-1$.
	In particular, $u_r$ therefore clearly generates $V_r$ as an $A$-module, which
	equals $D^{(r)}$ for $r \le p-2$, and it is free by Corollary \ref{freeness}.
	Every element of $D^{(p-1)}$ may then be written in the form $\pi^mu_{p-1}^a$
	with $m \in \zp$ and $a \in A$, and such an element can clearly only be trivial if $m$ is, and
	therefore $a$ is as well.  Noting that our choices of $\pi$ and $u_{p-1}$ as in Lemma \ref{pi}
	satisfy the desired relations, the presentation for $r = p-1$ is as stated.

	For $r = p$, the union of $\{1\}$ and the images of the functions $\phi_m$ and $\phi'_m$
	is the set of positive integers that are congruent to $1$ modulo $p-1$.
	Theorem \ref{main} and Theorem \ref{main1} imply that there
	exists an element of the $A$-module generated by $u_p$ and $w$ in $V_i(\xi)$ for each 
	$i \equiv 1 \bmod p-1$.  Thus, this $A$-module is $D^{(p)}$.
	Our choices of $u_p$ and $w$ satisfy the relations of Lemma \ref{defofelts}, and
	it follows from the second statement of Theorem \ref{main1} that if either
	$c \in A$ or $d \in \zp[\Phi]$ is nonzero, then so is $c u_p + d w$.
\end{proof}

\subsection{Refined elements} \label{refined}

In this section, we provide refinements of the elements constructed in Theorem \ref{main} and
Theorem \ref{main1}.
We maintain the notation of Section \ref{spec}.  We begin by constructing certain one-sided inverses to the
monotonically increasing functions $\phi$ and $\phi_m$.

For any nonnegative integer $a$ and positive integer $t$, let us set 
$$
	\langle a \rangle_t = \max(a + \{t-a\},t).
$$  
Therefore, $\langle a \rangle_t$ is the smallest integer greater than or equal to 
$t$ and $a$ and congruent to $t$ modulo $p-1$.
Define $\psi \colon \Z_{\ge 0} \to \Z_{\ge 0}$ by
\[
	\psi(a) = \left\lfloor\frac{\langle a \rangle_r +1}{p}\right\rfloor - \delta
\]
except for $r = p-1$ and $a \le p-1$, 
in which case we set $\psi(a) = 0$.
For $m \ge 0$, define $\psi_m \colon \Z_{\ge 0} \ra \Z_{\ge 0}$ by
\[
  	\psi_m(a) = \psi\left(\left\lceil\frac{a+1}{p^m}\right\rceil - 1 \right).
\]
Note that $\psi_0 = \psi$.

\begin{lemma} \label{phipsilem}
	We have $\psi_m(\phi_m(j)) = j$ for all nonnegative integers $j$.
	Moreover, for all such $j$ and positive integers $a$, we have
	$\phi_m(j) \ge a$ if and only if $j \ge \psi_m(a)$.
\end{lemma}

\begin{proof}
	First, note that $\phi(j)$ is congruent to $r$ modulo $p-1$, so we have
	$$
		\psi(\phi(j)) =  \left\lfloor\frac{\phi(j)+1}{p}\right\rfloor - \delta = 
		 \left\lfloor\frac{p(j+\delta)+\{r-\delta-j\}+1}{p}\right\rfloor-\delta= j
	$$
	unless $r \ge p-1$ and $j = 0$, but one checks immediately that $\psi(\phi(0)) = \psi(r) = 0$
	if $r \ge p-1$ as well.
	It follows that we have
	$$
		\psi_m(\phi_m(j)) = \psi\left(\left\lceil\frac{p^m(\phi(j)+1)}{p^m}\right\rceil - 1 \right)
		= \psi(\phi(j)) = j.
	$$
	Therefore, if $\phi_m(j) \ge a$,
	then $j = \psi_m(\phi_m(j)) \ge \psi_m(a)$, since $\psi_m$ is nondecreasing.  
	
	To finish
	the proof, we need only show that $\phi_m(\psi_m(a)) \ge a$, since $\phi_m$ is
	nondecreasing (in fact, strictly increasing).  First, note that the definition of $\psi$
	is such that
	$$
		\psi(a) = \psi(\langle a \rangle_r).
	$$
	For $i \equiv r \bmod p-1$ with $i \neq 1, p-1$, the value
	$\phi(\psi(i))$ is the unique integer between $p\lfloor
	\frac{i+1}{p}\rfloor$ and $p\lfloor\frac{i+1}{p}\rfloor + p-2$ which is congruent to $r \bmod
	p-1$.
	This implies that
	\begin{equation} \label{phipsi}
		\phi(\psi(a)) = \begin{cases}
			\langle a \rangle_r & \text{if } \langle a \rangle_r \not\equiv -1 \bmod p,
			\text{ or } a \le r = p-1, \\
	      		\langle a \rangle_r +p-1 & \text{otherwise.}
		\end{cases}
	\end{equation}	
	which is, in particular, at least $a$.
	By definition of $\phi_m$ and $\psi_m$, we then have that
	$$
		\phi_m(\psi_m(a)) = p^m(\phi(\psi_m(a))+1)-1 \ge 
		p^m\left \lceil \frac{a+1}{p^m} \right \rceil - 1 \ge a.
	$$
\end{proof}

We actually need a version of Lemma \ref{phipsilem} with $\phi_m$ replaced by $\phi'_m$ and
$\psi_m$ replaced by an appropriate function $\psi'_m \colon \Z_{\ge 0} \to \Z_{\ge 0}$, which we now define.
Set $\psi'_m = \psi_m$ if $r \le p-1$ and, if $r = p$, let
$$
	\psi'_m(a) = \begin{cases} \psi_m(a)-1 & \text{if }
	p^{m+l+1}+p^m \le a \le p^{m+l+1}+p^{m+1}-1  \text{ for some } l \ge 0, \\
	\psi_m(a) & \mr{otherwise.} \end{cases}
$$
Note that $\psi'_m(a) = \psi_m(a)-1$ for $r = p$ if and only if 
$\phi_m(p^l-1) < a \le \phi'_m(p^l-1)$ for some $l \ge 0$, in which case $\psi'_m(a) = p^l-1$.
One then easily checks the following.

\begin{corollary} \label{phipsiprime}
	We have $\psi'_m(\phi'_m(j)) = j$ for all nonnegative integers $j$.
	Moreover, for all such $j$ and positive integers $a$, we have
	$\phi'_m(j) \ge a$ if and only if $j \ge \psi'_m(a)$.
\end{corollary}

For the rest of this section, we fix a positive integer $i$ with $i \equiv r \bmod p-1$.

\begin{remark}
	Lemma \ref{phipsilem} and Theorem \ref{main} (resp., Corollary \ref{phipsiprime} 
	and Theorem \ref{main1})
	tell us that each $\alpha_{m,\psi_m(i)}$ (resp., $\beta_{m,l}$ with $\psi'_m(i) = p^l-1$)
	lies in $V_i$.   These elements have the form $(p^mbT^j + c)u_r + dw$ for $j = \psi'_m(i)$,
	where $b \in \zp[\Phi]-p\zp[\Phi]$,  $c \in T^{j+1}A$, and $d \in \Z[\Phi]$, with $d = 0$ if $r \neq p$.
	The same results also show that no such element with $j < \psi'_m(i)$ can lie in $V_i$.  
\end{remark}

For any $m \ge 0$, define $\theta_m \colon \Z_{\ge 1} \to \Z_{\ge 0}$ by 
$$
	\theta_m(a) = \psi\left(\left\lceil \frac{\langle a \rangle_r }{p^m} \right\rceil\right).
$$
By Lemma \ref{basic}a and Lemma \ref{phipsilem}, the value $\theta_m(i)$ for $i \equiv r \bmod p-1$ 
is the minimal integer $j$ such that
$p^m T^j u_r \in V_i$.  In particular, $\theta_m(i) \ge \psi_m(i)$ for all $i$.

\begin{lemma} \label{inequality}
	For all positive integers $m$ and $k$ with $k \le m$, we have
	$$
		\phi'_{k-1}(\psi'_m(i))-\delta \ge \theta_{m-k}(i)  - 1,
	$$
	with equality if and only if
	\begin{equation} \label{equivcond}
		p^{m-k+1}\phi'_{k-1}(\psi'_m(i)) < i.
	\end{equation}
	Moreover, we have $\psi'_m(i) \ge \theta_m(i) -1$, with 
	equality if and only if the above equivalent conditions hold for $k =1$.
\end{lemma}

\begin{proof}
	Let us check the case that $r = p$ and $\psi'_m(i) = p^l-1$ for some $l \ge 0$
	separately.  First, suppose that $p^{m+l+1} < i < p^{m+l+1}+p^{m+1}$.  
	In this case, we have
	$\psi'_m(i) = \theta_m(i)-1$.
	We also have
	$$
		\phi'_{k-1}(\psi'_m(i)) = \phi'_{k-1}(p^l-1) = p^{k+l}+p^k -1
		= \psi(p^{k+l+1}+p^{k+1}) \ge \theta_{m-k}(i),
	$$
	with equality if and only if
	\begin{equation} \label{primeineq}
		p^{m-k+1}\phi'_{k-1}(\psi'_m(i)) = p^{m+l+1}+p^{m+1}-p^{m-k+1} < i.
	\end{equation}
	Moreover, in the case that $p^{m+l+1} - p^{m+1} + 2p^m 
	< i \le p^{m+l+1}$, the values $\phi'_{k-1}(\psi'_m(i))$
	and $p^{m-k+1}\phi'_{k-1}(\psi'_m(i))$ are the same as in 
	the previous case, while $\theta_{m-k}(i)-1$ and $i$ are smaller.
	So, we may assume from this point forward that $r$ and $i$ are 
	such that $\psi'_m(i) = \psi_m(i)$
	and $\phi'_{k-1}(\psi'_m(i)) = \phi_{k-1}(\psi_m(i))$ for all $k$.

	We claim that $\rho^{m-k}T^{\phi_{k-1}(\psi_m(i))+1-\delta}u_r$ lies in $V_{i+p-1}$
	(resp., $\rho^{m-k}T\alpha_{k,\psi_m(i)}$ lies in $V_{i+p-1}$)  
	for all positive (resp., nonnegative) $k \le m$. 
	Note that $T\alpha_{m,\psi_m(i)} \in V_{i+p-1}$ as a consequence of Theorem \ref{main}.
	Suppose that $\rho^{m-k} T \alpha_{k,\psi_m(i)} \in V_{i+p-1}$ for some positive $k \le m$.  
	We then have that 
	$$
		\rho^{m-k} T^{\phi_{k-1}(\psi_m(i))+1-\delta} u_r
		 \sim -[r]! \rho^{m-k} T \alpha_{k,\psi_m(i)} \in V_{i+p-1},
	$$
	which also forces $\rho^{m-k+1} T\alpha_{k-1,\psi_m(i)} \in V_{i+p-1}$, since
	$$
		\rho^{m-k+1} T\alpha_{k-1,\psi_m(i)} = 
		\rho^{m-k} T\alpha_{k,\psi_m(i)} + \frac{1}{[r]!} \rho^{m-k}T^{\phi_{k-1}(\psi_m(i))+1-\delta} u_r,
	$$
	proving the claim.  In particular, since $\rho^m T\alpha_{0,\psi_m(i)}
	\in V_{i+p-1}$, we have $\rho^m T^{\psi_m(i)+1}u_r \in V_{i+p-1}$ as well.
	The definition of $\theta_{m-k}(i)$ now yields the desired inequalities.
			
	Equation \eqref{equivcond} holds for a given $k$
	if and only if $\rho^{m-k+1}\alpha_{k-1,\psi_m(i)} \notin V_i$.  Since
	$$
		[r]! \rho \alpha_{k-1,\psi_m(i)} \sim T^{\phi_{k-1}(\psi_m(i))-\delta}u_r,
	$$
	this occurs if and only if 
	$\rho^{m-k}T^{\phi_{k-1}(\psi_m(i))-\delta}u_r \notin V_i$
	and, therefore, if and only if 
	$$
		\phi_{k-1}(\psi_m(i))-\delta \le \theta_{m-k}(i)  - 1,
	$$
	which must then be an equality.  Note also that $\psi_m(i) < \theta_m(i)$
	if and only if $\rho^m\alpha_{0,\psi_m(i)} \notin V_i$, which holds by Lemma \ref{basic}a if and only if
	$p^m\phi(\psi_m(i)) < i$, the same condition as \eqref{equivcond} for $k = 1$.
\end{proof}

From now on, we set $i_m = \lceil \frac{i}{p^m} \rceil$ for all $m \ge 0$.

\begin{lemma} \label{ineqcond}
	For any pair of positive integers $m$ and $k$ with $k \le m$, we have
	$$
		\phi'_{k-1}(\psi'_m(i))-\delta \ge \theta_{m-k}(i)  - 1,
	$$
	with equality
	if and only if
	\begin{enumerate}
		\item $i_{m+\epsilon} \not\equiv 0 \bmod p$, or $r = p-1$
		and $i_m = p$,
		\item $i_{m+\epsilon} \equiv r +1\bmod p-1$, but not $r = p-1$ and $i_m =1$, and 
		\item $i \equiv - j \bmod p^{m+\epsilon}$ for some $0 < j < p^{m+1-k}$, 
	\end{enumerate}
	where $\epsilon = 0$ unless $r = p$ and $i_{m+1} = p^l+1$ 
	for some $l \ge 0$, in which case we
	set $\epsilon = 1$.
	Moreover, we have $\psi'_m(i) \ge \theta_m(i)-1$, with equality if and only 
	if the above conditions hold with $k = 1$.
\end{lemma}

\begin{proof}
	The case that $r = p$ and $\psi'_m(i) = p^l-1$ for some $l \ge 0$ follows from
	the proof of Lemma \ref{inequality}, noting that if $i_{m+1} = p^l+1$,
	then it is both nonzero modulo $p$ and congruent to $p+1$ modulo $p-1$, and
	the third condition of the lemma holds exactly when \eqref{primeineq} does.
	On the other hand, for the remaining $i$ with $\psi'_m(i) = p^l-1$, we have 
	$i_{m+1} = p^l$, and the fact that the inequality is strict was shown in the proof of
	 Lemma \ref{inequality}.
	So, we again assume that $r \neq p$ or $i$ is such that $\psi'_m(i) \neq p^l-1$
	for all $l \ge 0$.
	
	By Lemma \ref{inequality}, it suffices to determine the precise 
	conditions under which \eqref{equivcond}
	holds.  Let us set $a = \lceil \frac{i+1}{p^m} \rceil$.
	It follows from \eqref{phipsi} that we have 
	\begin{equation} \label{compute}
		p^{m-k+1}\phi_{k-1}(\psi_m(i)) = \begin{cases}
			p^m\langle a \rangle_{r+1}
			- p^{m-k+1} & \text{if } p \nmid \langle a \rangle_{r+1}, 
			\\
			p^m\langle a \rangle_{r+1}
			+ p^m(p-1) - p^{m-k+1} & \text{otherwise},
		\end{cases}
	\end{equation}
	unless $r = p-1$ and $\langle a \rangle_{r+1} = p$, in which case
	$p^{m-k+1}\phi_{k-1}(\psi_m(i)) = p^{m+1}-p^{m-k+1}$.
	Aside from this exceptional case, \eqref{compute} implies that $p$ cannot divide 
	$\langle a \rangle_{r+1}$ if \eqref{equivcond} is to hold.
	Moreover, if $\langle a \rangle_{r+1} > a$, then
	again \eqref{equivcond} cannot hold, so for it to hold, we must have 
	$a \equiv r+1 \bmod p-1$, but not $r = p-1$ and $a = 1$. 
	Assuming that these necessary conditions
	hold, the condition that
	$$
		p^{m-k+1}\phi_{k-1}(\psi_m(i)) = p^m a - p^{m-k+1} < i
	$$
	is exactly that $i \equiv -j \bmod p^m$ with $0 < j < p^{m-k+1}$.
\end{proof}

For $m \ge 0$, we will define new elements $\kappa_{m,i}$ of $V_i$ that involve fewer terms and easier-to-compute exponents of powers of $T$ than the expressions for $\alpha_{m,\psi_m(i)}$ and $\beta_{m,l}$.  In preparation, set 
$$
	\sigma(m,i) = \lfloor \log_p(p^m i_m - i) \rfloor
$$
for any $m \ge 0$ such that $p^m \nmid i$.  
Note that $0 \le \sigma(m,i) \le m-1$ when it is defined
and $\sigma(m+1,i)$ is defined and greater than or equal to $\sigma(m,i)$ whenever $\sigma(m,i)$
is defined.

First, supposing either that $r \le p-1$ or that $r = p$ and $i_{m+1} -1$ is not a power of $p$, we set
\begin{equation} \label{kappaone}
	\kappa_{m,i} =
	\rho^m T^{\theta_m(i)} u_r
\end{equation}
if $i_m \not\equiv r+1 \bmod p-1$, $p \mid i_m$, $i < p^m$, or $p^m \mid i$, unless $r = p-1$ and $i_m = p$,
and 
\begin{equation} \label{kappatwo}
	\kappa_{m,i} = \left(\rho^m T^{\theta_m(i)-1} 
	-a_{m,i}
	\sum_{k=\sigma(m,i)}^{m-1} \rho^k T^{\theta_k(i)-1}\right)u_r
\end{equation}
otherwise, where $a_{m,i}$ denotes the least positive residue of the inverse of 
$\{ r+1-\delta-\theta_m(i) \}!$ modulo $p$ unless $r = p-1$ and $\theta_m(i) = 1$, in
which case we take $a_{m,i} = -1$.
In the remaining case that $r = p$ and $i_{m+1}-1$ is a power of $p$, we set
\begin{equation} \label{kappathree}
	\kappa_{m,i} = \left(\rho^m T^{\theta_m(i)-1} 
	+ \sum_{k=\sigma(m+1,i)}^{m-1} \rho^k T^{\theta_k(i)-1}\right)u_r
	+ \rho^{m+\log_p(i_{m+1}-1)+1}w.
\end{equation}
For consistency, we let $a_{m,i} = -1$ for such $m$.  Note that Lemma \ref{ineqcond} tells us that 
each $\kappa_{m,i}$ has the form $(\rho^m T^{\psi'_m(i)} + c)u_r + dw$ for some
$c \in T^{\psi'_m(i)+1}A$ and $d \in \zp[\Phi]$, with $d$ taken to be zero if $r \le p-1$.

We give two examples for $p = 5$ and particular values of $i$.

\begin{example}
	Suppose that $p = 5$, $r = 3$, and $i = 11899$.  Then we have
	\begin{align*}
		\kappa_{0,i} &= T^{2380} u_3, & 
		\kappa_{1,i} &= \rho T^{476} u_3, \\
		\kappa_{2,i} &= (\rho^2 T^{95} - \rho T^{475} - T^{2379})u_3, &
		\kappa_{3,i} &= (\rho^3 T^{19} - \rho^2 T^{95}) u_3, \\
		\kappa_{4,i} &= \rho^4 T^4 u_3, &
		\kappa_{5,i} &= (\rho^5 - \rho^4 T^3 - \rho^3T^{19})u_3.
	\end{align*}
\end{example}

\begin{example}
	Suppose that $p = 5$, $r = 5$, and $i = 92729$.  Then we have
	\begin{align*}
		\kappa_{0,i} &= T^{18545} u_5, &
		\kappa_{1,i} &= (\rho T^{3708} - T^{18544})u_5,\\
		\kappa_{2,i} &= \rho^2 T^{741} u_5, &
		\kappa_{3,i} &= (\rho^3 T^{147} - \rho^2 T^{740} - \rho T^{3708})u_5, \\
		\kappa_{4,i} &= \rho^4 T^{29} u_5, &
		\kappa_{5,i} &= (\rho^5 T^4 + \rho^4 T^{28})u_5 + \rho^7 w, \\
		\kappa_{6,i} &= \rho^6 u_5 + \rho^7 w.
	\end{align*}
\end{example}

\begin{remark} \label{sigmarem}
	It is not hard to see from the definition of $\sigma(m,i)$ that $\sigma(m,i) \ge k$ for
	$k < m$ if and only if $p^{m-k} \nmid i_k$.
	Moreover, if for a given $k$ there exists $m > k$ such that
	$\sigma(m,i)$ is less than $k$ or not defined, then $p \mid i_k$ so
	$\kappa_{k,i} = \rho^kT^{\theta_k(i)}u_r$ unless $r = p$ and $i_{k+1}-1$ is a power
	of $p$ or $r = p-1$ and $i_k = p$.  The previous examples illustrate some of this.
\end{remark}

Let us show that the $\kappa_{m,i}$ are actually elements of $V_i$.  In the process, we see
how they compare to the elements $\alpha_{m,\psi_m(i)}$ and $\beta_{m,l}$ previously defined.

\begin{proposition}
	The elements $\kappa_{m,i}$ lie in $V_i$ for all nonnegative integers $m$.
\end{proposition}

\begin{proof}
	Suppose first that $r \neq p$ or $i$ does not satisfy 
	$i_m = p^l+1$ for any $l \ge 0$ (and omitting the case $r = p-1$ and $\psi_m(i) = 0$, for
	which one should take the fractions in the following two equations to be $1$). 
	If $\psi_m(i) = \theta_m(i)$, then we have that
	$$
		\kappa_{m,i} = \frac{[r]!}{\{r-\delta-\psi_m(i) \}!}\rho^m\alpha_{0,\psi_m(i)},
	$$
	and this lies in $V_i$ by  the definition of $\theta_m(i)$.
	If $\psi_m(i) = \theta_m(i)-1$, we claim that
	\begin{equation} \label{kappaclaim}
		\kappa_{m,i} \sim
		\frac{[r]!}{\{r-\delta-\psi_m(i) \}!}\rho^{\sigma(m,i)} \alpha_{m-\sigma(m,i),\psi_m(i)}.
	\end{equation}
	To see this, note that
	$$
		\kappa_{m,i} = \rho^{\sigma(m,i)} \left(\rho^{m-\sigma(m,i)} T^{\psi_m(i)} - 
		a_{m,i}\sum_{k=1}^{m-\sigma(m,i)}
		\rho^{m-\sigma(m,i)-k}T^{\theta_{m-k}(i)-1}\right)u_r.
	$$
	Lemma \ref{ineqcond} tells us that 
	$\theta_{m-k}(i)-1 = \phi_{k-1}(\psi_m(i))-\delta$ if and only if $p^{m-k+1} > p^m i_m - i$, and therefore 
	if $k \le m-\sigma(m,i)$, proving the claim.
	(Note that we the reason we do not have actual equality in \eqref{kappaclaim} 
	is simply that we took $a_{m,i}$ to be
	an inverse to $\{r-\delta-\psi_m(i)\}!$ modulo $p$, not in $\zp^{\times}$.)
	Moreover, we have by Theorem \ref{main} that  $\kappa_{m,i} \in V_t$ with
	$$
		t = p^{\sigma(m,i)}\phi_{m-\sigma(m,i)}(\psi_m(i)).
	$$
	Since $p^{\sigma(m,i)} \le p^m i_m - i$,
	Lemma \ref{ineqcond} implies that 
	$$
		\phi_{m-\sigma(m,i)}(\psi_m(i))-\delta \ge \theta_{\sigma(m,i)-1}(i),
	$$	
	and Lemma \ref{inequality} then states that $t \ge i$.
	
	Finally, if $r =p$ and $i_{m+1} = p^l+1$ for some $l \ge 0$, then Lemma \ref{ineqcond}
	similarly implies that
	$$
		\kappa_{m,i} = \rho^{\sigma(m+1,i)} \beta_{m-\sigma(m+1,i),l}.
	$$
	By Theorem \ref{main1}, we have in this case that 
	$\kappa_{m,i}  \in V_t$ 
	with
	$$	
		t = p^{\sigma(m+1,i)} \phi'_{m-\sigma(m+1,i)}(p^l-1) \ge i,
	$$
	the inequality again following from Lemmas \ref{inequality} and \ref{ineqcond}.
\end{proof}

\subsection{Generating sets}  \label{gen}

In this subsection, we give explicit minimal generating sets of all of the $A$-modules $V_i$ in terms of the
elements $\kappa_{m,i}$ of the previous section.  We
begin with generation.  Recall that $\delta \in \{0,1\}$ is $1$ if and only if $r = p$.

\begin{theorem} \label{gens}
	We let $S_i = \{ \kappa_{m,i} \mid 0 \le m \le s \}$
	for
	$$
		s = \left\lceil\log_p\left(\frac{i+1}{r+1+\delta(p-1)}\right)\right\rceil.
	$$
	If $2 \le r \le p-1$, then $S_i$ generates  $V_i$ as an $A$-module, while if
	$r = p$, then $S_i \cup   \{ p^{\lceil\log_p(i)\rceil} w \}$ generates $V_i$ as an $A$-module.
\end{theorem}

\begin{proof}
	Let $t = \lceil \frac{i+1}{p^m} \rceil -1$.
	In the case that $2 \le r \le p-1$, we have $\psi_m(i) = \psi(t)$ and 
	$\psi(t) > 0$ if and only if $\frac{i+1}{p^m} > r+1$, or $m < \log_p(\frac{i+1}{r+1})$.
	The smallest $m$ such that $\psi_m(i) = 0$ is therefore $s$.
	If $r = p$, then $\psi'_m(i) = \psi(t) - \epsilon_t$, where
	$\epsilon_t \in \{0,1\}$ is $1$ if and only if $p^{l+1}+1 \le t \le p^{l+1}+p-1$ for some $l \ge 0$.
	In particular, we have $\psi(t) > \epsilon_t$ if and only if $t \ge 2p$, so
  	the smallest $m$ such that $\psi'_m(i) = 0$ is again $s$.
	
  	Suppose that $\alpha = (\rho^k bT^j + c)u_r + dw \in V_i$ for some nonnegative integers $j$
  	and $k$, $b \in \zp[\Phi]-p\zp[\Phi]$, $c \in T^{j+1}A$, and $d \in \zp[\Phi]$ (with
	$d = 0$ if $r \neq p$).  Let $m = \min(k,s)$.  
	Then $j \ge \psi'_m(i)$ by Theorems \ref{main} and \ref{main1} and Corollary \ref{phipsiprime}
	(and the fact that $\psi'_s(i) = 0$), 
  	and we set
  	$$
    		\alpha' = \alpha-\rho^{k-m}  b
    		T^{j-\psi'_m(i)}\kappa_{m,i} \in V_i \cap (T^{j+1}Au_r + \zp[\Phi]w).
  	$$
	Since $\lim_{h \to \infty} T^h \kappa_{m,i} = 0$ for each $m \le s$, 
  	we may repeat this process recursively until we arrive at a limit of $0$ in the case that $r \le p-1$
	or an element of 
	$$
		V_i \cap \zp[\Phi]w = \zp[\Phi]p^{\lceil \log_p(i) \rceil}w
	$$
	if $r = p$, proving generation as an $A$-module.
\end{proof}

We will require the following lemma.

\begin{lemma} \label{thetaineq}
	If $m \ge 1$ is such that $\theta_m(i) \ge 1$, then 
	$\theta_{m-1}(i) \ge \theta_m(i) + 2$.
\end{lemma}

\begin{proof}
	First, note that $i_{m-1} \ge p(i_m-1) + 1$.  Therefore,
	\begin{equation} \label{thetam-1}
		\theta_{m-1}(i) \ge \psi(p(i_m-1)+1) = i_m-1 + \left\lfloor \frac{2 + \{r - i_m\}}{p} \right\rfloor - \delta.
	\end{equation}
	On the other hand,
	\begin{equation} \label{thetam}
		\theta_m(i) = \left\lfloor \frac{i_m+1 + \{r-i_m\}}{p} \right\rfloor - \delta.
	\end{equation}
	In particular, $\theta_m(i) = 1$ exactly when $r+1 \le i_m \le r+(\delta+1)(p-1)$.  
	In this case,
	$$
		\theta_{m-1}(i) \ge r + 1 -\delta \ge 3 = \theta_m(i) + 2.
	$$
	In general, \eqref{thetam-1} and \eqref{thetam} tell us that
	\begin{eqnarray*}
		\theta_{m-1}(i) \ge i_m-1 - \delta  &\mr{and}& \frac{i_m}{p}+1-\delta \ge \theta_m(i),
	\end{eqnarray*}
	and we have
	$$
		i_m-1-\delta \ge \frac{i_m}{p} + 3 - \delta
	$$
	if and only if $i_m \ge 4\frac{p}{p-1}$, which holds for $i_m \ge r+p$ unless $i_m = 5$, 
	$r = 2$, and $p=3$, in which case $\theta_{m-1}(i) \ge 5$ and $\theta_m(i) = 2$.
\end{proof}

For each $m \ge 0$, let us set $\epsilon_m(i) = \theta_m(i)-\psi'_m(i)$, which lies in $\{0,1\}$
by Lemma \ref{inequality} and the remark before it.
The following corollary is useful in understanding the form of our special elements.

\begin{corollary} \label{psiineq}
	For every $m \ge 0$, we have $\psi'_m(i) \ge \psi'_{m+1}(i)$, with equality if and only if
	$\psi'_m(i) = 0$.
\end{corollary}

\begin{proof}
	If $\theta_{m+1}(i) \ge 1$, Lemma \ref{thetaineq} and the fact that 
	$\epsilon_k(i) \in \{0,1\}$ for all $k$ imply that $\psi'_m(i) > \psi'_{m+1}(i)$.  
	Otherwise, $\psi'_{m+1}(i) = 0$, and the inequality holds automatically, with equality
	exactly if $\psi'_m(i) = 0$.
\end{proof}

We next show that the sets given in Theorem \ref{gens} are minimal unless $r = p$.  
It is in the proof of this result that the refined elements $\kappa_{m,i}$ first hold an
advantage of ease of use over the elements of Section \ref{spec}.
	
\begin{theorem} \label{mingens}
	For $r \le p-1$, no proper subset of $S_i$ generates $V_i$ as an $A$-module.
	For $r = p$, every proper subset of $S_i \cup \{ p^{\lceil \log_p(i) \rceil} w \}$ that
	generates $V_i$ as an $A$-module must contain $S_i$.
\end{theorem}	
	
\begin{proof}
	Assume first that $2 \le r \le p-1$.  
	Suppose that
	\begin{equation} \label{lindep}
		\sum_{m=0}^s c_m \kappa_{m,i} = 0,
	\end{equation}
	where $c_m \in A$ for $m \le s$.  We must show that
	no $c_m$ is a unit.  We prove the somewhat stronger claim that
	$c_m \in (p,T^{\epsilon_m(i)+1})$ for each $m$.  
 	
	Fix an nonnegative integer $m \le s$.
 	If $\epsilon_m(i) = 0$, then
	$\kappa_{m,i} = \rho^mT^{\theta_m(i)}u_r$ by  \eqref{kappaone}.
	If $\epsilon_m(i) = 1$, then \eqref{kappatwo} tells us that
	$$
		\kappa_{m,i} \equiv \rho^mT^{\theta_m(i)-1}u_r \bmod AT^{\theta_m(i)+1}u_r,
	$$
	noting Lemma \ref{thetaineq}.	
	Set 
	\begin{equation} \label{keyset}
		X_m = \{ k \in \Z \mid  m < k \le s,\, \epsilon_k(i) = 1,\, \sigma(k,i) \le m \}, 
	\end{equation}
	which is actually a set of cardinality at most one, though we do not need this fact.
	Suppose that $m < k \le s$.  If $k \in X_m$, then 
	\eqref{kappatwo} and Lemma \ref{thetaineq} together imply that
	$$
		\kappa_{k,i} \equiv -a_{k,i} \rho^mT^{\theta_m(i)-1}u_r \bmod (p^m,T^{\theta_m(i)+1})u_r,
	$$
	and if $k \notin X_m$, they and \eqref{kappaone} similarly imply that 
	$\kappa_{k,i} \in (p^m,T^{\theta_m(i)+1})u_r$.  
	Thus, \eqref{lindep} yields the congruence
	\begin{equation} \label{modm}
		c_m \rho^m T^{\psi_m(i)} \equiv \sum_{k \in X_m}
		c_k a_{k,i} \rho^m T^{\theta_m(i)-1} \bmod (p^{m+1},T^{\theta_m(i)+1}).
	\end{equation}
	If the claim holds for all $k > m$, then we have $c_k \in (p,T^2)$ for each $k \in X_m$, 
	so $c_m \in (p,T^{\epsilon_m(i)+1})$, as desired.
	
	If $r = p$, a completely 
	analogous argument shows that at most $p^{\lceil \log_p(i) \rceil}w$ is unnecessary 
	for generation, if one works modulo $Aw = \zp[\Phi]w + A(\varphi-1)u_p$ throughout.
	Here, one should replace $X_m$ by
	\begin{equation} \label{keyset2}
		X'_m = \{ k \in \Z \mid  m < k \le s,\, \epsilon_k(i) = 1,\, \sigma'(k,i) \le m \},
	\end{equation}
	where we set $\sigma'(k,i) = \sigma(k,i)$ unless $i_{k+1}=p^l+1$ for some $l \ge 0$,
	in which case we set $\sigma'(k,i) = \sigma(k+1,i)$.
\end{proof}

For purposes of completeness, we also give the precise condition on $i$ under which no proper subset
of $S \cup \{ p^{\lceil \log_p(i) \rceil} w \}$ generates $V_i$ in the case that $r = p$.  

\begin{proposition} \label{wneeded}
	For $r = p$, the set $S_i$ generates $V_i$ if and only if $i_s = p+1$.
\end{proposition}

\begin{proof}
	To determine whether $p^{\lceil \log_p(i) \rceil}w$ is or is not necessary, we work in distinct
	ranges of $i$ separately.  Note that the definition of $s$ forces $2p^s < i < 2p^{s+1}$.  
	
	\vspace{1ex}
	{\em Case 1:} $2p^s < i \le p^{s+1}$.  In this case,
	all of the elements $\kappa_{m,i}$ lie in $Au_p$, and therefore $p^{s+1}w$
	is necessary.  
	
	\vspace{1ex}
	{\em Case 2:} $p^{s+1} < i \le p^{s+1}+p^s-p^{s-1}$.  In this range, we have
	\begin{eqnarray*}
		\kappa_{s,i} = \rho^s u_p + \rho^{s+1} w &\text{and}&
		\kappa_{s-1,i} = \rho^{s-1} T^{p-1}u_p + \rho^{s+1}w.
	\end{eqnarray*}
	Note that
	$$
		(T-p)\kappa_{s,i} = \rho^s(T-p)u_p + \rho^{s+1}(\varphi-1)u_p
		= \rho^s(T -\rho)u_p,
	$$
	so 
	\begin{eqnarray} \label{redterms}
		\rho^sTu_p \equiv \rho^{s+1}u_p \bmod A\kappa_{s,i}
		&\mr{and}& \rho^s u_p \equiv -\rho^{s+1}w \bmod A\kappa_{s,i}.
	\end{eqnarray}
	Applying these to $\rho\kappa_{s-1,i}$, we obtain
	$$
		\rho\kappa_{s-1,i} \equiv \rho^{s+p-1}u_p + \rho^{s+2}w 
		\equiv (-\rho^{s+p} + \rho^{s+2})w \bmod A\kappa_{s,i},
	$$
	which in particular tells us that $p^{s+2}w \in A(\kappa_{s-1,i},\kappa_{s,i})$.
	
	\vspace{1ex}
	{\em Case 3:} $p^{s+1} +p^s -p^{s-1} < i \le p^{s+1}+p^s$.  In this range, we have
	$\kappa_{s,i} = \rho^s u_p + \rho^{s+1} w$
	and
	$$
		\kappa_{s-1,i} = \left(\rho^{s-1} T^{p-1} +
		\sum_{k=\sigma(s,i)}^{s-2} \rho^k T^{\theta_k(i)-1}\right)u_p + \rho^{s+1}w
	$$
	with $\sigma(s,i) \le s-2$.
	Moreover, $\kappa_{m,i} \in Au_r$ for $m \le s-2$. 
	
	Set $\nu_{k,i} = \rho^k T^{\theta_k(i)} u_r$ for all nonnegative $k$.	  
	We note that $\nu_{m,i} \in A(\kappa_{0,i}, \ldots, \kappa_{m,i})$ for $m \le s-2$: 
	if $\kappa_{m,i} \neq \nu_{m,i}$, which is to say $\epsilon_m(i) = 1$, then 
	$$
		\nu_{m,i} = T\kappa_{m,i} + a_{m,i} \sum_{k=\sigma(m,i)}^{m-1} \nu_{k,i}.
	$$ 
	
	Let $j = \theta_{s-2}(i)-p$, and note that $j \ge p^2-1 \ge 2$.  
	Since 
	$$
		T^j\kappa_{s-1,i} \equiv \rho^{s-1} T^{\theta_{s-2}(i)-1} u_p
		+ \rho^{s+1}T^j w \bmod A(\nu_{\sigma(s,i),i}, \ldots, \nu_{s-2,i})
	$$
	and $\rho^{k+1}T^{\theta_k(i)-1}u_p \in A\nu_{k+1,i}$ 
	for all $k$ with $\sigma(s,i) \le k \le s-3$, we therefore have
	that 
	\begin{equation} \label{intermed}
		(\rho-T^j)\kappa_{s-1,i} \equiv 
		\rho^sT^{p-1}u_p + \rho^{s+1}(\rho-T^j) w 
		\bmod A(\nu_{\sigma(s,i),i}, \ldots, \nu_{s-2,i}).
	\end{equation}
 	Using \eqref{redterms} to reduce \eqref{intermed}, we see that
	$$
		(\rho-T^j)\kappa_{s-1,i} \equiv \rho^{s+2}(1-\rho^{p-2}-\rho^{j-1})w
		\bmod A(\nu_{\sigma(s,i),i}, \ldots, \nu_{s-2,i},\kappa_{s,i}),
	$$
	which implies that $p^{s+2}w \in A(\kappa_{0,i}, \ldots, \kappa_{s,i})$.
	
	\vspace{1ex}
	{\em Case 4:} $p^{s+1}+p^s < i < 2p^{s+1}$.  In this case, all of the $\kappa_{m,i}$
	with $m \le s-1$ lie in $AT^p u_p$, and so for $p^{s+2}w$ to be unnecessary,
	there would have to exist $c \in A$ such that
	\begin{equation} \label{notposs}
		c \kappa_{s,i} \equiv p^{s+2} w \bmod AT^2 u_p.
	\end{equation}
	Note that
	$$
		c \kappa_{s,i} \equiv c (\rho^s u_p + \rho^{s+1} w) \bmod AT^2 u_p,
	$$
	and this forces $c = T^2 c'$ for some $c' \in A$.  This means that 
	$$
		c\kappa_{s,i} \equiv c' p^2\rho^{s+1} w \bmod Au_p 
	$$
	but $p^{s+2}w \notin A(p^{s+3}w,u_p)$, so \eqref{notposs} cannot hold.
\end{proof}

\section{The finite level} \label{finlevel}
\subsection{Norms and eigenspace structure} \label{eigfin}

In this section, we explore the consequences of the results of Section \ref{inflevel} for unit groups of actual abelian local fields
of characteristic $0$.  Fix a positive integer $n$.  
Recall from the introduction that $F_n$ is the field obtained from $E$ by adjoining the
$p^n$th roots of unity and that $U_{n,t}$ denotes the $t$th unit group of 
$F_n$ for $t \ge 1$.  As before, we set  $\Gamma_n = \Gal(F_n/F_1)$.

For positive integers $m \ge n$, let $N_{m,n}$ and $\Tr_{m,n}$ denote, respectively, the norm and trace from $F_m$ to $F_n$.  We also let $N_n$ denote the restriction map $N_n \colon F^{\times} \to F_n^{\times}$
on norm compatible sequences.  Recall that $\lambda_n = N_n(\lambda) = 1-\zeta_{p^n}$, where
$\zeta_{p^n} = N_n(\zeta)$ is a primitive $p^n$th root of unity.  We will require a few preliminary lemmas.

\begin{lemma} \label{trace}
	One has
	$$
		\Tr_{n+1,n}(\lambda_{n+1}^{pk-\epsilon}) \equiv p\lambda_n^{k-\epsilon} \bmod p^3
	$$
	for all $k \ge 1$ and $\epsilon \in \{0,1\}$.
\end{lemma}

\begin{proof}
	An easy calculation shows that
	$$
		\Tr_{n+1,n}(\lambda_{n+1}^t) = p \sum_{j=0}^{\lfloor \frac{t}{p} \rfloor} 
		\binom{t}{pj} (-\zeta_{p^n})^j
	$$ 
	for every $t \ge 0$.  Since
	$$
		\binom{pk-\epsilon}{pj} = \binom{k-\epsilon}{j} \prod_{\substack{ s=1\\ p \nmid s}}^{p(k-j)}
		\left( 1+ \frac{pj}{s} \right) \equiv \binom{k-\epsilon}{j} \bmod p^2
	$$
	for any $j \ge 0$, we have the result.
\end{proof}

Let $e_n = p^{n-1}(p-1)$ denote the ramification index of $F_n$.
In applying Lemma \ref{trace}, it is useful to make note of the fact that
\begin{equation} \label{paspow}
	p \equiv -\lambda_n^{e_n} \bmod \lambda_n^{p^n}.
\end{equation}

\begin{lemma} \label{normelt}
	For $t \ge 1$ and any unit $\eta$ in $E$, one has
	$$
		N_{n+1,n}(1+\eta \lambda_{n+1}^t) \equiv \begin{cases}
			1 + \eta^p \lambda_n^t \bmod \lambda_n^{t+1} & \text{if } t < p^n-1, \\
			1 + (\eta^p-\eta) \lambda_n^{p^n-\epsilon} \bmod \lambda_n^{p^n+1-\epsilon} & \text{if } t = p^n-\epsilon,\,
			\epsilon \in \{0,1\}, \\
			1 - \eta \lambda_n^{e_n + k-\epsilon} \bmod \lambda_n^{e_n+k+1-\epsilon} 
			& \text{if } t = pk-\epsilon > p^n,\, \epsilon \in \{0,1\}.
		\end{cases}
	$$
	Moreover, we have 
	$$
		N_{n+1,n}(1+\eta\lambda_{n+1}^t) \equiv 1 \bmod \lambda_n^{e_n+ \lfloor \frac{t}{p} \rfloor}.
	$$
	for all $t > p^n$.
\end{lemma}

\begin{proof}
	The jump in the ramification filtration of $\Gal(F_{n+1}/F_n)$ occurs at $p^n-1$.
	By \cite[Lemmas V.4 and V.5]{serre},
	we have
	$$
		N_{n+1,n}(1+\eta\lambda_{n+1}^t) \equiv 1+ \eta\Tr_{n+1,n}(\lambda_{n+1}^t)
		+ \eta^p \lambda_n^t \bmod \lambda_n^{e_n + \lfloor 2t/p \rfloor}
	$$
	and
	$$
		\Tr_{n+1,n}(\lambda_{n+1}^t) \equiv 0 \bmod \lambda_n^{e_n + \lfloor t/p \rfloor}.
	$$
	The result is then a corollary of Lemma \ref{trace}, upon applying \eqref{paspow}.
\end{proof}

Let $D_n$ be the pro-$p$ completion of $F_n^{\times}$, and let $D_n^{(r)} = D_n^{\varepsilon_r}$ for any 
$r \in \Z$.  As before, we fix $r$ with $2 \le r \le p$, and $i$ will always denote a positive integer with 
$i \equiv r \bmod p-1$.  Let $V_{n,i} = U_{n,i}^{\varepsilon_r} = U_{n,i} \cap D_n^{(r)}$ for any such $i$.  
These $V_{n,i}$ are all modules over $A_n = \zp[\Gamma_n \times \Phi]$.  
As in Lemma \ref{eigmod}, we have isomorphisms
$$
	V_{n,i}/V_{n,i+p-1} \xrightarrow{\sim} \F_q
$$ 
that send $1+x\lambda_n^i$ for some $x$ in the valuation ring of $F_n$ to the element $\bar{x}$
of $\F_q$ that is identified with the image of $x$ in the residue field of $F_n$ under the isomorphism
fixed in Section \ref{prelim}.  We may then set
$V_{n,i}' = V_{n,i} - V_{n,i+p-1}$ and define $V_{n,i}(\eta)$ for $\eta \in {\bf F}_q^{\times}$ as the
set of elements $1+x\lambda_n^i$ with $\bar{x} = \eta$.

We have the following consequence of Lemma \ref{normelt}.

\begin{lemma} \label{normlem}
	For any $t \ge -1$, we have
	$$
		N_{n+1,n}(V_{n+1,p^n+t}) \subseteq V_{n,p^n+t-(p-1)\lfloor \frac{t+1}{p} \rfloor},
	$$
	with equality for $t \ge 0$.
\end{lemma}

\begin{proof}
	Note that Lemma \ref{normelt} tells us that 
	$$
		N_{n+1,n}(U_{n+1,p^n+pk-\epsilon}) = U_{n,p^n+k-\epsilon}
	$$
	for all $k \ge 0$ and $\epsilon \in \{0,1\}$ with $k \ge \epsilon$, since every element in
	$U_{n,p^n+k-\epsilon}$ can be written as a product of elements of the
	form $1+\eta_t\lambda_n^{p^n+t}$ with $t \ge k-\epsilon$ and $\eta_t \in \F_q$.
	(For $k = 0$ and $\epsilon = -1$, it tells us just that any element of $U_{n+1,p^n-1}$ has a norm
	in $U_{n,p^n-1}$.)
	
	Note that 
	\begin{eqnarray*} 
		U_{n,p^n+k-\epsilon}^{\varepsilon_r} = V_{n,p^n+k-\epsilon+\{r-k+\epsilon-1\}} &\text{and}&
		U_{n+1,p^n+pk-\epsilon}^{\varepsilon_r} = V_{n,p^n+pk-\epsilon+\{r-k+\epsilon-1\}}.
	\end{eqnarray*}
	For any $t \ge 0$, we may write $t = pk-\epsilon+\{r-k+\epsilon-1\}$ 
	for some $k$, $\epsilon$, and $r$, and we have
	$$
		t-(p-1)\left\lfloor \frac{t+1}{p} \right\rfloor = k-\epsilon+\{r-k+\epsilon-1\},
	$$
	hence the result.
\end{proof}

The next corollary is almost immediate from Lemmas \ref{normelt} and \ref{normlem}, so we
leave it to the reader.

\begin{corollary} \label{normcor}
	For any unit $\eta$ in $E$, one has
	$$
		N_{n+1,n}(V_{n+1,i}(\eta)) \subseteq 
		\begin{cases}
			V_{n,i}(\eta^p) & \text{if } i < p^n-1,\\
			V_{n,i}(\eta^p-\eta) & \text{if } i = p^n-1, \\
			V_{n,p^n+k-1}(-\eta) & \text{if } i = p^n+pk-1 \text{ for some } k > 0,
		\end{cases}
	$$
	with equality if $r \neq p-1$ or $i > p^n$.
\end{corollary}

As for the $p$-power map, we have the following well-known and easy-to-prove fact. 

\begin{lemma} \label{power}
	Suppose that $i > p^{n-1}$.  Then the $p$th power map induces an isomorphism
	$V_{n,i} \xrightarrow{\sim} V_{n,i+e_n}$, and we have $V_{n,i}(\eta)^p = V_{n,i+e_n}(-\eta)$
	for all $\eta \in \F_q^{\times}$.
\end{lemma}

Next, we discuss the restriction map from the field of norms to the finite level.

\begin{proposition} \label{restr}
	The map $N_n$ induces maps $N_n \colon D^{(r)} \to D_n^{(r)}$ that are
	surjections for $r \neq p-1$ and which have procyclic cokernel for $r = p-1$. 
	For $\eta \in \F_q^{\times}$, we have
	$$
		N_n(V_i(\eta)) \subseteq \begin{cases} 
			V_{n,i}(\eta^{p^{-n}}) & \mr{if}\ i \le p^n-2, \\
			V_{n,p^n-1}(\eta^{p^{-n}}-\eta^{p^{-n-1}}) & \mr{if}\ i = p^n-1, \\
			V_{n,p^n+k-1}(-\eta^{p^{-n-1}}) & \mr{if}\ i = p^n+pk-1 \text{ for some } 0 < k < e_n.
		\end{cases}
	$$
	Moreover, we have induced maps
	$$
		V_i/V_{i+1} \to V_{n,i}/V_{n,i+1}
	$$
	for all $i < p^n$, 
	and these are isomorphisms for $i \neq p^n-1$.
	For $i \le p^n$, we have that $V_{n,i} = N_n V_i$ if $r \neq p-1$,
	and $V_{n,i}/N_nV_i$ is procyclic if $r = p-1$.
\end{proposition}

\begin{proof}
	That the cokernel of $N_n$ is trivial (resp., procyclic) if $r \neq p-1$ (resp., $r = p-1$) follows
	easily from local class field theory, but it is also a consequence of the argument that follows.
	The first jump in the ramification filtration of $\Gal(F_{\infty}/F_{n+1})$ 	
	is at $p^{n+1}-1$.  In particular, for $t$ less than this value, repeated application of Lemma 
	\ref{normelt} tells us that
	$$
		N_{n+1}(1+\eta\lambda^t) = \lim_{m \to \infty} N_{m,n+1}(1+\eta^{p^{-m}}\lambda_m^t)
		\equiv 1 + \eta^{p^{-n-1}} \lambda_{n+1}^t \bmod
		\lambda_{n+1}^{t+1}.
	$$
	Moreover, repeated application of Corollary \ref{normcor} followed by two applications of 
	Lemma \ref{normlem} tells us that 
	$$
		N_n(V_{p^{n+1}-1+\{r\}}) \subseteq V_{n,p^n+e_n-1+\{r\}}.
	$$
	An application of Corollary \ref{normcor} then yields the stated containments.
	
	Since $\eta^{p^{-n}}$ and $-\eta^{p^{-n-1}}$ run through all elements of $\F_q$ as
	as $\eta \in \F_q$ varies, we obtain $V_{n,i} = N_n V_i + V_{n,i+p-1}$ for all $i \le p^n$ but $p^n-1$.
	Noting Lemma \ref{power}, this implies 
	$$
		V_{n,i+ke_n} = N_n V_{p^k i} + V_{n,i+ke_n+p-1}
	$$ 
	for $p^{n-1} < i \le p^n$ with $i \neq p^n-1$ and $k \ge 0$.  
	Note that every element of every $V_{n,i}$ may be written as an infinite product over $j \ge 0$
	of one element from each of a fixed set of representatives of the $V_{n,i+j(p-1)}/V_{n,i+(j+1)(p-1)}$.  
	Thus, we have that $N_nV_i = V_{n,i}$ so long as $r \neq p-1$. 
	
	If $r = p-1$, we can choose an element $z_n$ of $V_{n,p^n-1}(\xi)$
	that is not a norm.  By the above-proven formula for 
	$N_n(1+\eta \lambda^{p^n-1})$ modulo $\lambda_n^{p^n}$, we have that
	$$
		V_{n,p^n+ke_n-1} = N_n V_{p^k (p^n-1)} + V_{n,p^n+ke_n+p-2} + \langle z_n^{p^k} \rangle
	$$ 
	for $k = 0$, and then for all $k \ge 0$ by taking powers.  Therefore, $V_{n,i}/N_nV_i$ is generated
	by $z_n$ for all $i < p^n$ with $i \equiv 0 \bmod p-1$.
\end{proof}

The following structural result is again essentially found in \cite{greither}, without
the stated congruences.  Here, we derive it from more basic principles.

\begin{theorem} \label{presfin}
	For $r \le p-2$, 
	the $A_n$-module $D_n^{(r)}$ is freely generated as an $A_n$-module by an element
	$u_{n,r} \in V_{n,r}(\xi)$.
	The $A_n$-module $D_n^{(p-1)}$ has a presentation
	$$
		D_n^{(p-1)} = \langle \pi_n, u_{n,p-1}, v_n \mid 
		\pi_n^{\varphi} = \pi_n, \, \pi_n^{\gamma-1} = u_{n,p-1}^{N_{\Phi}}, \,
		v_n^{\gamma} = v_n, \, u_{n,p-1}^{N_{\Gamma_n}} = v_n^{1-\varphi} \rangle,
	$$
	where $v = v_n^{\varphi^{2-n}} \equiv 1 + p\xi \bmod p^2$ is independent of $n$ and
	$u_{n,p-1} \in V_{n,p-1}(\xi)$ for $n \ge 2$, while $u_{1,p-1} \in V_{n,p-1}(\xi-\xi^{p^{-1}})$. 
	The $A_n$-module $D_n^{(p)}$ has a presentation
	$$
		D_n^{(p)} = \langle u_{n,p}, w_n \mid w_n^{\gamma-1-p} =
    		u_{n,p}^{\varphi-1} \rangle
	$$
	with $u_{n,p} \in V_{n,p}(\xi)$ and $w_n \in V_{n,1}(-\xi)$ such that 
	$w_n^{N_{\Phi}} = \zeta_{p^n}$.
\end{theorem}

\begin{proof}
	We set $u_{n,r} = (N_n u_r)^{\varphi^n}$, $\pi_n = N_n \pi$, and $w_n = (N_n w)^{\varphi^n}$ 
	with $u_r$, $\pi$, and 
	$w$ as in Theorem \ref{presentation}.  
	It follows from the surjectivity of $N_n$ for $r \neq p-1$ in Proposition \ref{restr}  that the element
	$u_{n,r}$ generates $D_n^{(r)}$ for $r \le p-2$, while the elements $w_n$ and $u_{n,r}$ generate
	$D_n^{(p)}$.  By Hilbert's Theorem 90, the kernel of $N_n$ consists exactly of elements of the
	form $\alpha^{\gamma^{p^{n-1}}-1}$ with $\alpha \in D$, and therefore it follows that
	$D_n^{(r)}$ is free of rank $1$ on $u_{n,r}$ over $A_n$ for $r \le p-2$ and that $D_n^{(p)}$ has the
	stated presentation.  (That $u_{1,p} \in V_{1,p}(\xi)$ requires a simple check using 
	Propositions \ref{newgen} and \ref{restr}.)
	
	The elements $\pi_n$ and $u_{n,p-1}$ automatically satisfy the first two relations in the
	desired presentation of $D_n^{(p-1)}$.  In particular,
	$$
		u_{n,p-1}^{N_{\Gamma_n \times \Phi}} = \pi_n^{(\gamma-1)N_{\Gamma_n}} = 1,
	$$
	so Hilbert's Theorem 90 tells us that $u_{n,p-1}^{N_{\Gamma_n}} = v_n^{1-\varphi}$ for some 
	$v_n$ in the pro-$p$ completion of $E^{\times}$.  By Proposition \ref{restr}, we have that
	$$
		u_{n,p-1}^{N_{\Gamma_n}} 
		\equiv 1+ (\xi^{p^{n-1}}-\xi^{p^{n-2}})\lambda_1^{p-1} \bmod \lambda_1^p.
	$$
	Noting \eqref{paspow},
	we may in fact choose $v_n \equiv 1+ p\varphi^{n-2}(\xi)  \bmod p^2$ with $v = v_n^{\varphi^{2-n}}$
	independent of $n$.
	
	HiIbert's theorem 90 and Theorem \ref{presentation} tell us that the
	$A_n$-module generated by $u_{n,p-1}$ is isomorphic to $A_n/(N_{\Gamma_n \times \Phi})$.
	By Proposition \ref{restr}, the cokernel of $N_n$ on $D^{(p-1)}$ is isomorphic to $\zp$.  
	We claim that the image of $v$ topologically generates this cokernel. 
	If this is the case, then clearly $D_n^{(p-1)}$ is generated by $\pi_n$, $u_{n,p-1}$, and $v$, and  
	any solution with $b, d \in \zp$ and $c \in A_n$ to
	$\pi_n^b u_{n,p-1}^c v^d = 1$ must satisfy $b = d = 0$ and $c \in \zp N_{\Gamma_n \times \Phi}$.  
	
	It remains only to demonstrate the claim.  Suppose by way of contradiction that there exists 
	$a \in A_n$ such that
	$x = vu_{n,p-1}^a$ is a $p$th power in $D_n^{(p-1)}$.  This implies that
	$x^{\gamma-1} = u_{n,p-1}^{a(\gamma-1)}$ is a $p$th power in the $A_n$-module generated
	by $u_{n,p-1}$. 
	It follows that
	$a(\gamma-1) \in A_n(p,N_{\Gamma_n \times \Phi})$, which forces $a(\gamma-1) \equiv 
	0 \bmod p$, so $a \in A_n(p,N_{\Gamma_n})$.  It then suffices to show that
	$$
		v u_{n,p-1}^{bN_{\Gamma_n}} = v^{1+b\varphi^{n-2}(1-\varphi)}
	$$
	is not a $p$th power in $F_n$ for any $b \in \zp[\Phi]$.  If it were for some $b$, then $v^{N_{\Phi}}$ and
	hence $1+p$ would be a $p$th power in $F_n$ as well, but this is clearly not the case.
\end{proof}

\subsection{Special elements} \label{specfin}

We assume for the rest of the paper that $n \ge 2$, the case that $n =1$ being slightly exceptional 
but also completely straightforward.
In this subsection, we construct special elements in the groups in the unit filtration of 
$F_n^{\times}$.  
Aside from the case that $r = p-1$, these arise as restrictions of the elements introduced in
Section \ref{spec}.

Note that 
$$
	\zp[\Gamma_n] \cong \zp[T]/(f_n),
$$ 
where $f_n = (T+1)^{p^{n-1}}-1$.  Of course, we can then speak of the action of $T$ on an element of $D_n^{(r)}$.
Once again reverting to additive notation, the following is now an immediate corollary of Theorem \ref{main} and Proposition \ref{restr}.

\begin{proposition} \label{mainfin}
  	Let $m$ and $j$ be nonnegative integers with $\phi_m(j) < p^n-1$.
  	Define
  	\begin{equation*}
    		\alpha_{n,m,j} = \frac{1}{[r]!}
    		\left(\{r-\delta-j\}!\rho^mT^j -  \sum_{k=1}^{m} \rho^{m-k}T^{\phi_{k-1}(j)-\delta}\right)u_{n,r},
  	\end{equation*}
	unless $j = 0$ and $r = p-1$, in which case we replace $\{r-\delta-j\}!$ with $-1$
	in the formula.
  	Then $\alpha_{n,m,j} \in V_{n,\phi_m(j)}(\xi)$.  Furthermore,
  	$$
    		(p^mbT^j + c)u_{n,r} \notin V_{n,\phi_m(j)+p-1}
  	$$
  	for all $b \in \zp[\Phi]-p\zp[\Phi]$ and $c \in T^{j+1}A_n$.
\end{proposition}

For nonnegative $m \le n-2$, 
define $\phi'_{n,m} \colon \Z_{\ge 0}
\to \Z_{\ge 0}$ by $\phi'_{n,m}(j) = \phi'_m(j)$ unless $r = p-1$ and $j = e_{n-m-1}$,
in which case we set
$$
	\phi'_{n,m}(e_{n-m-1}) = e_n + p^{m+1} - 1 = \phi_m(e_{n-m-1})+p^m(p-1).
$$
For nonnegative $k$, define
$\vartheta_{2,k} = 1+ \varphi^{-1} + \cdots + \varphi^{-k}$ and
$\vartheta_{j,k} = 1$ for $j > 2$.  
Note that $\vartheta_{2,k} \in p\zp[\Phi]$ if and only if $k \equiv -1 \bmod p|\Phi|$.

By Theorem \ref{presfin}, every element of $V_{n,p-1}$
may be written as $cu_{n,p-1} + dv$ with $c \in A_n$ and $d \in \zp$, and this representation is unique up to the choice of $c$ modulo $N_{\Gamma_n \times \Phi}$.  For $a, b \in D_n^{(r)}$, we again write $a \sim b$ if $a, b \in V_{n,i}(\eta)$ for some $i$ and $\eta \in \F_q^{\times}$.

\begin{theorem}\ \label{mainr0fin}
	Let $m \le n-2$ be a nonnegative integer, and define
	$$
		\omega_{n,m} = \sum_{k=0}^m \rho^{m-k}\vartheta_{n-m,k}
		T^{p^{n-m+k-2}(p-1)+p^k-1}u_{n,p-1} - v.
	$$
 	Then we have $\omega_{n,m} \in V_{n,e_n+p^{m+1}-1}(\xi)$. 
	Furthermore, if $j \ge 0$ with $\phi_m(j) < p^n$, then
	$$
		(p^m T^j b+c)u_{n,p-1} + dv \notin V_{\phi'_{n,m}(j)+p-1}
	$$
	for all $b \in \zp[\Phi]-p\zp[\Phi]$, $c \in T^{j+1}A_n$, and $d \in \zp$.
\end{theorem}

\begin{proof}	
	Let $l$ be a nonnegative integer with $l \le m$.  We define
	\begin{equation} \label{omega}
		\omega_{n,m,l} =  \sum_{k=0}^l \rho^{m-k}\vartheta_{n-m,k}
		T^{p^{n-m+k-2}(p-1)+p^k-1}u_{n,p-1} - v.
	\end{equation}
	We claim not only that $\omega_{n,m} = \omega_{n,m,m} \in V_{n,e_n+p^{m+1}-1}(\xi)$, but that 
	$$
		\omega_{n,m,l} \in \begin{cases} V_{n,e_n+p^{m+1}-p^{m-l}}(\vartheta_{n-m,l+1}\xi)
		& \text{if } p \nmid \vartheta_{n-m,l+1}, \\
		V_{n,e_n+p^{m+1}-p^{m-l-1}}(\xi)
		& \text{if } p \mid \vartheta_{n-m,l+1}
		\end{cases}
	$$
	for $l < m$.
	We note, to begin with, that $\omega_{n,m,l} \in V_{n,e_n+p-1}$, since Lemma \ref{basic}a 
	implies 
	$$
		\omega_{n,m,l} + v \sim  \rho^m T^{e_{n-m-1}}u_{n,p-1} \in V_{n,e_n}(-\xi).
	$$
	For a given $i$, we take $V_{n,i}(0)$ to mean $V_{n,i+p-1}$ in what follows.

	If $p \nmid \vartheta_{n-m,l}$, then
	Lemmas \ref{basic}a and \ref{power} imply that
	$$
		T\omega_{n,m,l} \sim 
		\rho^{m-l}\vartheta_{n-m,l}T^{p^{n-m+l-2}(p-1)+p^l}u_{n,p-1}
	$$
	if $l < m$, $m = 0$, or $m < n-2$,  
	and we have
	\begin{equation} \label{Tomega}
		T\omega_{n,m,l} \in \begin{cases}
			V_{n,e_n+p^{m+1}+p^{m-l}(p-2)}(-\xi) & \text{if } m < n-2 \text{ or } 
			m = 0, \\
			V_{n,p^n+p^{m-l-1}(p-2)}(\vartheta_{2,l}\varphi^{-1}\xi) & \text{if } l < m = n-2.
		\end{cases}
	\end{equation}
	On the other hand, if $p \mid \vartheta_{n-m,l}$, then we have $T\omega_{n,m,l} \sim T
	\omega_{n,m,l-1}$, so \eqref{Tomega} still holds.
	Moreover, since $\vartheta_{2,n-3} \varphi^{-1} - \vartheta_{2,n-2} = -1$ for $n \ge 3$, 
	we have that
	$$
		T \omega_{n,n-2} \sim \rho \vartheta_{2,n-3} T^{p^{n-3}}u_{n,p-1} + 
		\vartheta_{2,n-2} T^{p^{n-2}}u_{n,p-1} \in V_{n,p^n+p-2}(-\xi).
	$$

	We prove our claim 
	by induction on $m$.
	For $m = 0$, we have
	$T\omega_{n,0} \in V_{n,e_n+2p-2}(-\xi)$ by \eqref{Tomega}, and we have seen that
	$\omega_{n,0} \in V_{n,e_n+p-1}$, so Proposition \ref{repeat} forces  
	$\omega_{n,0} \in V_{n,e_n+p-1}(\xi)$.
	For $m \ge 1$, that $\omega_{n,m} \in V_{n,e_n+p^{m+1}-1}$ follows from the claim
	for $l = m-1$ and the fact that
	$$
		\omega_{n,m} -\omega_{n,m,m-1} = \vartheta_{n-m,m}T^{e_{n-1}+p^m-1}u_{n,p-1}
	$$
	is an element of $V_{n,e_n+p^{m+1}-p}(-\vartheta_{n-m,m}\xi)$.
	Since $T\omega_{n,m} \in V_{n,e_n+p^{m+1}+p-2}(-\xi)$,
	an application of Proposition \ref{repeat} would then yield that 
	$\omega_{n,m} \in V_{n,e_n+p^{m+1}-1}(\xi)$.
	So, to perform the inductive step for $l < m$, we assume that
	either $p \nmid \vartheta_{n-m,l+1}$ or 
	$l = m-1$, since otherwise $\omega_{n,m,l} \sim \omega_{n,m,l+1}$ and $l + 1 < m$.
	
	By Lemma \ref{power} and induction, we have
	\begin{equation} \label{omnorm}
		N_{n,n-1}(\omega_{n,m,l}) = p\omega_{n-1,m-1,l} \in 
		\begin{cases} V_{n-1,2e_{n-1} + p^m - p^{m-l-1}}(-\vartheta_{n-m,l+1}\xi) & \text{if }
		l < m-1,
		\\
		V_{n-1,2e_{n-1} + p^m - p^{m-l-1}}(-\xi) & \text{if } l = m-1.
		\end{cases}
	\end{equation}
	Let $i$ be such that $\omega_{n,m,l} \in V'_{n,i}$, and set $t = e_n+p^{m+1}-p^{m-l}$.
	By Lemma \ref{normlem}, we have both that
	$i \le t+p-1$
	and that there exists 
	$x \in V'_{n,t}$	with $N_{n,n-1}(x) = p\omega_{n-1,m-1,l}$.
	Hilbert's Theorem 90 implies that $x - \omega_{n,m,l} \in A_nf_{n-1}u_{n,p-1}$.
	Note that 
	$$	
		pf_{n-1}u_{n,p-1} \sim pT^{p^{n-2}}u_{n,p-1} \in V_{n,p^n+p-2},
	$$
	while $\omega_{n,m,l} \notin V_{n,p^n+p-2}$. 
 	It follows that
	\begin{equation} \label{formofx}
		x \sim \omega_{n,m,l} + bT^g u_{n,p-1},
	\end{equation}
 	for some $b \in A_n$ with $b \notin (p,T)$ and $g \ge p^{n-2}$.
	Since $bT^{g+1}u_{n,p-1} \in V'_{n,\phi(g+1)}$ by Lemma \ref{basic}a and both
	$Tx$ and $T\omega_{n,m,l}$ lie in $V_{n,t+p-1}$, the latter by
	\eqref{Tomega}, we have that $\phi(g+1)>t$
	and hence that $\phi(g) \ge t$.
	Therefore, we have
	$bT^g u_{n,p-1} \in V_{n,t}$,
	and \eqref{formofx} now forces $i \ge t$, which means that $i \in \{t,t+p-1\}$.
	
	If $l < n-3$, then Lemma \ref{moving} forces $i = t$ in order for \eqref{Tomega} to hold.
	If $l = n-3$ and $i = t+p-1$, then Proposition \ref{repeat} and \eqref{Tomega} 
	force
	$\omega_{n,n-2,n-3} \in V_{n,p^n-1}(-\vartheta_{2,n-3}\varphi^{-1}\xi)$.
	By Corollary \ref{normcor}, this implies that
	$$
		N_{n,n-1}(\omega_{n,n-2,n-3}) \in V_{n-1,2e_{n-1}+p^{n-2}-1}(\vartheta_{2,n-3}\varphi^{-1}\xi),
	$$ 
	and then \eqref{omnorm} tells us that $p \mid \vartheta_{2,n-2}$ and 
	$\omega_{n,n-2,n-3} \in V_{n,p^n-1}(\xi)$.
	
	If $i = t$, then Lemma \ref{basic}a implies that
	\begin{equation} \label{omegaTpow}
		\omega_{n,m,l} \sim -dT^{e_{n-1}+p^m-p^{m-l-1}}u_{n,p-1}
	\end{equation}
	for some $d \in \zp[\Phi]-p\zp[\Phi]$.
	Set
	$$
		z = \omega_{n,m,l} + dT^{e_{n-1}+p^m-p^{m-l-1}}u_{n,p-1} \in V_{n,t+p-1}.
	$$
	By \eqref{Tomega} and Lemma \ref{basic}a, we have 
	$Tz \in V_{n,t+2(p-1)}(-d'\xi)$, where $d' = d$ if $l < n-3$ and 
	$d' = d - \vartheta_{2,n-3}\varphi^{-1}$ if $l = n-3$.  We therefore have $z \in V_{n,t+p-1}(d'\xi)$
	and then
	$$
		N_{n,n-1}(z) \in V_{n-1,2e_{n-1}+p^m-p^{m-l-1}}(-d'\xi)
	$$
	by Corollary \ref{normcor}.  On the other hand, we have
	$$
		N_{n,n-1}(T^{e_{n-1}+p^m-p^{m-l-1}}u_{n,p-1}) =
		\varphi T^{e_{n-1}+p^m-p^{m-l-1}}u_{n-1,p-1} \in V_{n-1,t},
	$$
	so we have
	$N_{n,n-1}(z) \sim N_{n,n-1}(\omega_{n,m,l})$.
	Equation \eqref{omnorm} then forces $d = \vartheta_{n-m,l+1}$.
	If $p \mid \vartheta_{n-m,l+1}$, then $l=n-3$ by assumption, and
	this contradicts our assumption on $i$ and implies the claim for $\omega_{n,n-2,n-3}$.  
	Otherwise, we have already shown that 
	$i = t$, and \eqref{omegaTpow} and Lemma \ref{basic}a yield the claim. 
	
	\vspace{1ex}
	Suppose now that $j \ge 0$, $b \in \zp[\Phi]-p\zp[\Phi]$, $c \in T^{j+1}A_n$, and $d \in \zp$
	are such that $\phi_m(j) < p^n$ and
	$$
		\omega = (p^m bT^j+c)u_{n,p-1} +dv \in V'_{n,i}
	$$
	for some $i \ge \phi'_{n,m}(j)$.  We suppose that $\phi_m(j) \ge e_n$, as
	the result otherwise reduces to Proposition \ref{mainfin}.	
	For $m = 0$, if
	$(bT^j +c)u_{n,p-1} \not\sim -dv$, then $i = e_n$ or $i = \phi(j) \le \phi'_{n,0}(j)$.  
	Otherwise, we must have $j = e_{n-1}$, and since $T\omega \sim bT^{j+1}u_{n,p-1}$,
	the argument of Lemma \ref{basic}b tells us that $i = e_n+p-1$.
	
	For $m \ge 1$, we rewrite $c$ as $pc' + T^h\nu$ for some $h \ge j+1$ and
	$c', \nu \in A_n$ with $\nu \notin (p,T)$.  
	Note that $\phi'_{n,m}(j) = p\phi'_{n-1,m-1}(j) +p-1$.
	By induction, we have that
	$$
		(p^{m-1} bT^j + c')\varphi u_{n-1,p-1} + dv \notin V_{n-1,\phi'_{n-1,m-1}(j)+p-1}.
	$$
	The $p$th power of this element is the norm from $F_n$ of
	$$
		\omega' = \omega -T^h \nu u_{n,p-1} = (p^m bT^j + pc')u_{n,p-1} + dv,
	$$
	and $\omega' \notin V_{n,\phi'_{n,m}(j)+p-1}$  by Lemma \ref{normlem}.  
	If $\omega' \in V_{n,\phi'_{n,m}(j)}$, then
	the fact that $\phi'_{n,m}(j)$ is $-1$ modulo $p$ and therefore not a value of $\phi$
	implies that $\omega' \not\sim -T^h \nu u_{n,p-1}$, so we have $i = \phi'_{n,m}(j)$.  
	
	So, assume that $\omega' \notin
	V_{n,\phi'_{n,m}(j)}$.  Then $\omega' \sim -T^h \nu u_{n,p-1}$,  
	and Lemma \ref{basic}a implies that $\phi(h) < \phi'_{n,m}(j) \le i$.	
	If $\omega \notin V_{n,\phi(h+1)}$, then we must have
	$i = \phi(h)+p-1 = \phi'_{n,m}(j)$.
	So, we assume moreover that $\omega \in V_{n,\phi(h+1)}$, in which case
	$T\omega' \sim -T^{h+1}\nu u_{n,p-1}$.  Since $T\omega'$ 
	is a power of $p$, either $\phi(h+1)$ is divisible by $p$ and less than $p^n$, 
	or $\phi(h+1) > p^n$.  In the former case, unless $\phi(h+2) > p^n$, we would 
	have $T^2\omega' \in V_{n,\phi(h+1)+p(p-1)}$ and then
	$T^2\omega \in V'_{n,\phi(h+2)}$, contradicting $\omega \in V_{n,\phi(h+1)}$.
	We therefore have $\phi(h+2) > p^n$ in both cases, so $T\omega' \in V_{n,p^n-p}$. 
	By Proposition \ref{mainfin}
	and the fact that
	$\phi_m(p^{n-m-1}) > p^n$, 
	this forces $j = p^{n-m-1}-1$.  If $m < n-2$, then 
	$$
		p^n-p \le \phi(h)+p-1 \le \phi'_{n,m}(j) = \phi_m(j) = p^n-p^{m+1}+p^m-1,
	$$ 
	which is a contradiction. 
	We therefore have $m = n-2$ and $j = p-1$, so
	$$
		p^n-1 = \phi'_{n,n-2}(p-1) > \phi_{n-2}(p-1),
	$$ 
	which, noting Proposition \ref{mainfin},
	implies that $p \nmid d$ and then, noting Theorem \ref{presfin}, that $\omega \notin pD_n^{(p-1)}$. 
	In particular, $\omega \notin V_{n,p^n+p-2}$, so $i = p^n-1$.
\end{proof}

\begin{remark}
	Note that $\phi_{n-1}(0) = p^n-1 < p^n$ as well, but in this case, the element
	$u_{n,p-1}^{N_{\Gamma_n \times \Phi}} = 1$ 
	has the form $(p^{n-1}b + c)u_{n,p-1}$ with $b \in \zp[\Phi] - p\zp[\Phi]$ and $c \in TA$.
\end{remark}

For $r = p$, we have the following consequence of Theorem \ref{main1} and Proposition \ref{restr}.

\begin{proposition} \label{mainr1fin}
  	Let $m$ and $l$ be nonnegative integers with $\phi_m(p^l-1) \le p^n$.
  	Let
  	\begin{equation*}
		\beta_{n,m,l} = 
   		   \left (\rho^m T^{p^l-1} + \sum_{k=1}^{m} \rho^{m-k}
    		T^{\phi'_{k-1}(p^l-1)-1}\right)u_{n,p} + \rho^{m+l+1}w_n.
  	\end{equation*}
	Then $\beta_{n,m,l} \in V_{n,\phi'_m(p^l-1)}(-\xi)$ unless $l = n-1$ and $m = 0$,
	in which case $\beta_{n,0,n-1} \in V_{n,p^n}(\xi^{p^{-1}})$.
  	Furthermore, for any $j \ge 0$ with $\phi_m(j) \le p^n$, we have
  	$$
    		(p^mbT^j + c)u_{n,p} + d w_n \notin V_{n,\min(\phi'_m(j)+p-1,p^n+p-1)}
  	$$
	for all $b \in \zp[\Phi]-p\zp[\Phi]$, $c \in T^{j+1}A_n$ and $d \in \zp[\Phi]$.
\end{proposition}

\subsection{Generating sets} \label{genfin}

In this final subsection, we turn to the task of finding small generating sets for the groups $V_{n,i}$ as $A_n$-modules.  First, we define the refined elements that will be used in forming these sets.

Suppose that $i \le p^n$ and 
$$
	0 \le m \le \left\lceil\log_p\left(\frac{i+1}{r+1+\delta(p-1) }\right)\right\rceil.
$$
Aside from the case that $r = p-1$ and $p^m < i-e_n < p^{m+1}$, we set
$$
	\kappa_{n,m,i} = \varphi^n N_n \kappa_{m,i},
$$
which can be written down explicitly as in the formulas 
\eqref{kappaone}, \eqref{kappatwo}, and \eqref{kappathree}, but now with 
$u_r$ replaced by $u_{n,r}$ 
and $w$ replaced by $w_n$.  By Propositions \ref{mainfin} and \ref{mainr1fin}, we have
that $\kappa_{n,m,i} \in V_{n,i}$.

If $r = p-1$ and $p^m < i - e_n < p^{m+1}$, 
then we set
$$
	\kappa_{n,m,i} = \omega_{n,m,m-\sigma(m+1,i)}
$$
with $\omega_{n,m,l}$ for $l \ge 0$ defined as in \eqref{omega}.
Then $\kappa_{n,m,i} \in V_i$ by the claim in the proof of Theorem \ref{mainr0fin}.
Moreover, we have
\begin{equation} \label{kappa4}
	\kappa_{n,m,i} = \sum_{k=\sigma(m+1,i)}^m \rho^k\vartheta_{n-m,m-k} 
	T^{\theta_k(i)-1} u_{n,p-1} - v,
\end{equation}
since $\theta_k(i) = p^{n-k-2}(p-1) + p^{m-k}$  if $k \ge \sigma(m+1,i)$. 

Our next result is the analogue of Theorem \ref{gens} at the finite level.

\begin{theorem} \label{gensfin}
	Let $\mu$ be the
	smallest nonnegative integer for which
	$i \le  \mu e_n + p^n$.
	Let 
	$$
		S_{n,i} = \{ p^\mu\kappa_{n,m,i-\mu e_n} \mid  0 \le m \le s \},
	$$
	where 
	$$
		s = \left\lceil\log_p\left(\frac{i-\mu e_n+1}{r+1+\delta(p-1) }\right)\right\rceil.
	$$ 
	If $2 \le r \le p-2$, then the $A_n$-module $V_{n,i}$ is 
	generated by $S_{n,i}$.
	If $r = p-1$, it is generated by 
	$S_{n,i} \cup \{ p^{\mu} v \}$ if $i \le (\mu+1)e_n$
	and $S_{n,i}$ otherwise, and if $r = p$, it is generated by
	$S_{n,i} \cup \{ p^{\mu+\lceil\log_p(i-\mu e_n)\rceil}w_n \}$.
\end{theorem}

\begin{proof}
	Suppose first that $i \le p^n$.  If $r \ne p-1$, then $V_{n,i} = N_nV_i$ by Proposition \ref{restr}.
	For such $i$, the generation then follows immediately from Theorem \ref{gens}.
	
	Similarly, if $r = p-1$, then $v \in V_{n,e_n}(-\xi)$ generates the cokernel of $N_n$.
	If $i \le e_n$, then $S_{n,i} \cup \{ v \}$ generates $V_{n,i}$ by a similar argument 
	to that given in Theorem \ref{gens} (or by Proposition \ref{restr} and Theorem \ref{gens} itself).  
	If $e_n < i < p^n$, then similarly $S_{n,i} \cup \{ pv \}$
	generates $V_{n,i}$, but we now claim that $pv$ is in the $A_n$-submodule generated by $S_{n,i}$.
	To see this, suppose that $m \le n-2$ is such that $p^m < i-e_n < p^{m+1}$.  
	Note that $A_nS_{n,i}$ contains $
	\nu_{n,k,i} = \rho^k T^{\theta_k(i)} u_{n,p-1}$
	for each $0 \le k \le n-1$. (If $\kappa_{n,k,i}$ is not this element, 
	one can multiply it by $T$ and subtract
	off multiples of the $\nu_{n,h,i}$ for $h < k$ to reduce it to this form.)  Noting \eqref{kappa4},
	we have 
	$$
		\rho v = -\rho\kappa_{n,m,i} +  \sum_{k=\sigma(m+1,i)}^{m}
		\vartheta_{n-m,m-k} T^{\theta_k(i)-\theta_{k+1}(i)-1} \nu_{n,k+1,i} \in A_nS_{n,i}.
	$$
	In the case of arbitrary $r$ and $i$, Lemma \ref{power} tells us that
	$V_{n,i} = p^{\mu} V_{n,i-\mu e_n}$, and we again have the desired generation.
\end{proof}

\begin{remark}
	For $i \le p^n$, the integer $s$ in Theorem \ref{gensfin} is unique such that $i$ lies in the 
	half-open interval
	$[(r+1)p^{s-1},(r+1)p^s)$ if $r \le p-1$ and $[2p^s, 2p^{s+1})$ if $r = p$. 
	Since $S_{n,i}$ has $s+1$ elements, the generating set 
	$S'_{n,i}$ provided in Theorem \ref{gensfin} has at most $n+1$ elements.  Since
	$S'_{n,i} = p^{\mu}S'_{n,i-\mu e_n}$, the 
	latter statement holds for all $i$.  In fact,
	for $i > p^{n-1}$,
	the set $S'_{n,i}$ has either $n$ or $n+1$ elements, depending for each $r$ on which of 
	two ranges $i$ lies in modulo
	$e_n$.
\end{remark}

Finally, we prove a slightly weaker minimality statement than Theorem \ref{mingens}, since
in the finite case there are many values of $i$ for which the analogous statement to Theorem 
\ref{mingens} is simply not true, so long as $r \le p-1$.

\begin{theorem} \label{mingensfin}
	Every generating subset of the generating set for $V_{n,i}$ of 
	Theorem \ref{gensfin} is of cocardinality at most one.
\end{theorem}

\begin{proof}
	We maintain the notation of Theorem \ref{gensfin}.
	By Lemma \ref{power}, the $p^{\mu}$th power map defines an isomorphism 
	$V_{i-\mu e_n} \xrightarrow{\sim} V_i$, and $S_{n,i} = p^{\mu}S_{n,i-\mu e_n}$.
	We therefore assume that $i \le p^n$ for the rest of the proof.
	Note that we have
	\begin{equation} \label{thetamax}
		\theta_k(i) \le p^{n-k-1}
	\end{equation}
	for all $0 \le k \le n-1$, and we have $\theta_n(i) = 0$.
	
	\vspace{1ex}
	{\em Case $r \le p-2$:} 
	In this case, $N_n$ induces an isomorphism $D^{(r)}/f_nD^{(r)}
	\xrightarrow{\sim} D_n^{(r)}$,
	so Proposition \ref{restr} tells us that $V_{n,i} \cong V_i/(V_i \cap f_nD^{(r)})$.
	In other words, a subset $Y_n$ of $S_{n,i}$ will generate $V_{n,i}$ if and only if the
	subset $Y$ of $S_i$ lifting it has the property that $Y \cup  \{f_n u_r\}$ generates
	$V_i + f_nD^{(r)}$.
	
	Recall that 
	$$
		f_n \equiv \sum_{k=0}^{n-1} p^k T^{p^{n-k-1}} 
		\bmod (p^{n-1}T^2, p^{n-2} T^{2p}, \ldots, T^{2p^{n-1}}).
	$$ 	
	Noting \eqref{thetamax}, we have that
	\begin{equation} \label{equivf}
		f_n \equiv p^mT^{p^{n-m-1}} \bmod (p^{m+1},T^{\theta_m(i)+1})
	\end{equation}
	for each $0 \le m \le s$.  Let us set 
	$I = (p,T,\varphi-1)$ and
	$$
		I_m = (p,T^{1+\epsilon_m(i)},\varphi-1)
	$$
	for the remainder of the proof.
	
	The analogue of \eqref{lindep} in our current setting is
	\begin{equation} \label{lindep2}
		\sum_{m=0}^s c_m \kappa_{m,i} = b f_n u_r
	\end{equation}
	for some $c_m \in A$ and $b \in A$.
	Given a solution to \eqref{lindep2}, we claim that there exist $q_k \in \zp$ 
	for $k \le s$, independent of the solution, such that
	\begin{equation} \label{cmcong}
		c_k \equiv  q_k bT^{\epsilon_k(i)} \bmod I_k.
	\end{equation}
	Of course, only those $\kappa_{n,k,i}$ for $k$ such that $p \nmid q_k$ and $\epsilon_k(i) = 0$
	can possibly be $A_n$-linear combinations of the others.  If $k$ is such a value and we
	suppose that $c_k = 0$, then these congruences force $b \in I$ and
	therefore $c_m \in I$ for every other $m \le s$, proving the result.
	
	We turn to the proof of the claim.
	In our current setting, equation \eqref{modm} becomes
	$$
		c_n \rho^n \equiv 0 \bmod (p^{n+1},T)
	$$
	for $m = n$ (if $s = n$, since $\theta_n(i) = 0$) and
	\begin{equation} \label{red2}
		c_m \rho^mT^{\psi_m(i)} \equiv  \sum_{k \in X_m} c_k a_{k,i} \rho^m T^{\theta_m(i)-1} 
		+ bp^mT^{p^{n-m-1}} \bmod (p^{m+1}, T^{\theta_m(i)+1})
	\end{equation}
	for $m \le n-1$, with $X_m$ as in \eqref{keyset}. 
	In the case that $s=n$, the claim for $k = n$ is then immediate.
	Moreover, supposing that we know the claim for $k$ with $m+1 \le k \le s$, the congruence
	\eqref{red2} implies that
	$$
		c_m \equiv \sum_{k \in X_m} q_ka_{k,i}b T^{\epsilon_m(i)}
		+ bT^{p^{n-m-1}-\psi_m(i)} \bmod I_m
	$$  
	upon application of \eqref{cmcong} for $k \in X_m$.  As $\epsilon_m(i) \le p^{n-m-1}-\psi_m(i)$
	by \eqref{thetamax}, we have the claim for $k = m$ as well.  
	
	We remark that if $\theta_m(i) < p^{n-m-1}$ for all $m \le n-1$, 
	which is to say that $i \le p^{n-1} r$, then we obtain recursively that $p \mid q_m$ 
	for all $m \le s$.  In other words,  $S_{n,i}$ has no proper
	generating subset for such $i$.  This is useful in the following case.
	
	\vspace{1ex}
	{\em Case $r = p$:} 
	In the case $r = p$, we have that $\theta_m(i) < p^{n-m-1}$ for all $m \le n-1$ and all $i \le p^n$
	(since $\delta = 1$), and	
	the analogous argument working modulo $Aw$ and using the set $X'_m$ of
	\eqref{keyset2} shows that any subset of
	$S_{n,i} \cup \{ p^{\lceil \log_p(i) \rceil} w_n \}$ that generates $V_{n,i}$
	must contain $S_{n,i}$.

	\vspace{1ex}
	{\em Case $r = p-1$:} Finally, we consider the more subtle case that $r = p-1$.  
	In this case, $s \le n-1$.  Recall from Theorem \ref{presfin} that
	$$
		V_{p-1}/f_nV_{p-1} \cong A_n u_{n,p-1} \cong A_n/(N_{\Gamma_n \times \Phi})
	$$
	and $A_n v = \zp v + \zp[\Phi] N_{\Gamma_n} u_{n,p-1}$.
	Note that $N_{\Gamma_n}$ lifts to $T^{-1}f_n$ in $A$.  As in \eqref{equivf}, we have
	$$
		T^{-1} f_n \equiv p^mT^{p^{n-m-1}-1} \bmod (p^{m+1},T^{\theta_m(i)+1})
	$$
	for $0 \le m \le n-2$ and
	\begin{equation} \label{Tinvfn}
		T^{-1}f_n \equiv p^{n-1}(1-\tfrac{1}{2}T) \bmod (p^n,T^2).
	\end{equation}
	
	\vspace{1ex} \noindent
	{\em Range $i \le e_n$:} In this range, every $\kappa_{n,m,i}$ lies
	in $A_n u_r$, so $v$ is in particular necessary to generate $V_{n,i}$.  
	We also have $\theta_{n-1}(i) = 0$ and $\theta_m(i) \le e_{n-m-1}$ for all $m \le n-2$.
	Consider the following analogue of \eqref{lindep}:
	\begin{equation} \label{lindep3}
		\sum_{m=0}^s c_m \kappa_{m,i} = bT^{-1} f_n u_{p-1}.
	\end{equation}
	As before, we claim that there exist $q_k \in \zp$ for $k \le s$, 
	independent of the solution to \eqref{lindep3}, such that \eqref{cmcong} holds,
	from which the result follows in this range.	
	
	The analogue of 
	\eqref{modm} for $m \le s$ in the current setting is
	\begin{equation} \label{red3}
		c_m \rho^mT^{\psi_m(i)} \equiv  \sum_{k \in X_m} c_k a_{k,i} \rho^m T^{\theta_m(i)-1} 
		+ bp^mT^{p^{n-m-1}-1} \bmod (p^{m+1}, T^{\theta_m(i)+1}).
	\end{equation}
	If $s = n-1$, we then obtain $c_{n-1} \equiv b \bmod I$.  If $s \ge n-2$, we have
	$$
		c_{n-2} \equiv bT^{p-1-\theta_{n-2}(i)+\epsilon_{n-2}(i)} \bmod I_{n-2},
	$$
	hence the claim for $k=n-2$. 
	For $m \le n-3$, we have $\theta_m(i) \le p^{n-m-1}-2$, and assuming the claim for 
	$m+1 \le k \le s$, we see recursively using \eqref{red3} that
	$$
		c_m \equiv \sum_{k \in X_m} q_k a_{k,i} b T^{\epsilon_m(i)} 
		 \bmod I_m.
	$$

	\vspace{1ex} \noindent
	{\em Range $e_n < i < p^n$:}  In this range, $s = n-1$, $\theta_{n-1}(i) = 1$, and $\theta_{n-2}(i) = p$.  
	Let $l \le n-2$ be such that $p^l < i - e_n < p^{l+1}$, 
	so $\kappa_{n,l,i}$ is the lone element of $S_{n,i}$ that does not lie in $A_n u_{n,p-1}$.    
	Thus, if we were to have
	\begin{equation} \label{lindepfin}
		\sum_{m=0}^{n-1} d_m\kappa_{n,m,i} = 0
	\end{equation}
	for some $d_m \in A_n$, then we would have to have $d_l \in A_n(T,\varphi-1)$ in order that
	$d_l\kappa_{n,l,i} \in A_n u_{n,p-1}$.  
	Let 
	$$
		\kappa'_{l,i} = \sum_{j=\sigma(l+1,i)}^l \rho^j \vartheta_{n-l,l-j} T^{\theta_j(i)} u_{p-1}
	$$
	so that
	$T\kappa_{n,l,i} = \varphi^n N_n \kappa'_{l,i}$.  Let $\kappa'_{m,i}
	= \kappa_{m,i}$ for $m \le n-1$ with $m \neq l$.  
	
	Equation \eqref{lindepfin} implies that
	\begin{equation} \label{lindep4}
		\sum_{m=0}^{n-1} c_m\kappa'_{m,i} \equiv bT^{-1}f_n u_{p-1} \bmod A(\varphi-1)u_{p-1}
	\end{equation}
	for some $b \in A$ and where 
	$c_m \in A$ reduces to $d_m$ for $m \neq l$ and $c_l \in A$ is such that 
	$Tc_l$ reduces to $d_l$ modulo $A_n(\varphi-1)$. 	
	Similarly to before, we claim that there exist $q_m \in \zp$ for $m \le n-2$, 
	independent of the solution to \eqref{lindep4}, such that
	\eqref{cmcong} holds, and that $b \in I$ if and only if $c_{n-1} \in I$.
	From this, it follows that a solution to \eqref{lindep4}
	with $c_k = 0$ for some $k$ has $c_m \in I$ for every other $m \le n-1$.
	
	Note that $\epsilon_l(i) = 0$, and
	let $\tau_m$ be $\vartheta_{n-l,l-m}$ if $\sigma(l+1,i) \le m < l$
	and $0$ otherwise.
	Equations \eqref{Tinvfn} and \eqref{lindep4} yield
	\begin{equation} \label{cn-1}
		c_{n-1}\equiv b(1-\tfrac{1}{2}T) \bmod (p,T^2,\varphi-1),
	\end{equation}
	and, for arbitrary $m \le n-2$, we have
	\begin{equation} \label{red4}
		c_m T^{\psi_m(i)} \equiv  \sum_{k \in X_m} c_k a_{k,i} T^{\theta_m(i)-1} 
		- \tau_m c_l T^{\theta_m(i)} + bT^{p^{n-m-1}-1} 
		\bmod (p,T^{\theta_m(i)+1},\varphi-1).
	\end{equation}
	
	For $m = n-2$, note that \eqref{cn-1}, \eqref{red4}, and $a_{n-1,i} = -1$ imply that
	\begin{equation} \label{cn-2}
		c_{n-2}T^{1-\epsilon_{n-2}(i)} \equiv b - c_{n-1} \equiv \tfrac{1}{2}bT
		\bmod (p,T^2,\varphi-1),
	\end{equation}
	so \eqref{cmcong} holds with $q_{n-2} = \tfrac{1}{2}$.
	For $m$ with $\sigma(n-1,i) \le m \le n-3$ (which exists only if $l = n-2$), 
	we have $X_m = \{ n-1 \}$ and $\theta_m(i) = p^{n-m-1}$,
	and we obtain from \eqref{red4} and \eqref{cn-2} that
	\begin{equation} \label{cmIm}
		c_m T^{1-\epsilon_m(i)} \equiv -c_{n-1} - c_{n-2}\tau_mT + b \equiv 
		\tfrac{1}{2}(1-\vartheta_{2,n-m-2})bT \bmod (p,T^2,\varphi-1),
	\end{equation}
	so \eqref{cmcong} holds with $q_m = -\tfrac{1}{2}(n-m-2)$.
	For $m < \sigma(n-1,i)$, we have $\theta_m(i) < p^{n-m-1}$,
	and \eqref{red4} and \eqref{cn-2} yield recursively that
	$$
		c_m \equiv \sum_{k \in X_m} q_ka_{k,i}b T^{\epsilon_m(i)} 
		- q_l \tau_m b T^{\epsilon_m(i)}
		+ bT^{p^{n-m-1}-1-\psi_m(i)} 
		\bmod I_m,
	$$ 
	verifying \eqref{cmcong} for $k = m$. 
\end{proof}

\emph{}

\end{document}